\documentclass[a4paper]{amsart}
\usepackage{amssymb, enumitem}
\usepackage{tikz}
\usepackage[all]{xy}

\usepackage{hyperref, aliascnt}
\usepackage{mathtools}
\def\today{\number\day\space\ifcase\month\or   January\or February\or
   March\or April\or May\or June\or   July\or August\or September\or
   October\or November\or December\fi\   \number\year}

\theoremstyle{definition}
\newtheorem{lma}{Lemma}[section]

\newaliascnt{thmCt}{lma}
\newtheorem{thm}[thmCt]{Theorem}
\aliascntresetthe{thmCt}

\newaliascnt{corCt}{lma}

\aliascntresetthe{corCt}

\newaliascnt{propCt}{lma}
\newtheorem{prop}[propCt]{Proposition}
\aliascntresetthe{propCt}

\newtheorem*{thm*}{Theorem}
\newtheorem*{qst*}{Question}
\newtheorem*{cor*}{Corollary}
\newtheorem*{prop*}{Proposition}

\newtheorem{thmx}{Theorem}

\newcounter{theoremintro}

\newaliascnt{pgrCt}{lma}

\aliascntresetthe{pgrCt}

\newaliascnt{dfCt}{lma}
\newtheorem{df}[dfCt]{Definition}
\aliascntresetthe{dfCt}

\newaliascnt{remCt}{lma}
\newtheorem{rem}[remCt]{Remark}
\aliascntresetthe{remCt}

\newaliascnt{remsCt}{lma}

\aliascntresetthe{remsCt}

\newaliascnt{egCt}{lma}

\aliascntresetthe{egCt}

\newaliascnt{egsCt}{lma}

\aliascntresetthe{egsCt}

\newaliascnt{qstCt}{lma}

\aliascntresetthe{qstCt}

\newaliascnt{pbmCt}{lma}

\aliascntresetthe{pbmCt}

\newaliascnt{notaCt}{lma}
\newtheorem{nota}[notaCt]{Notation}
\aliascntresetthe{notaCt}

\newcommand{\beq}{\begin{equation}}
\newcommand{\eeq}{\end{equation}}
\newcommand{\beqa}{\begin{eqnarray*}}
\newcommand{\eeqa}{\end{eqnarray*}}
\newcommand{\bal}{\begin{align*}}
\newcommand{\eal}{\end{align*}}
\newcommand{\bi}{\begin{itemize}}
\newcommand{\ei}{\end{itemize}}
\newcommand{\be}{\begin{enumerate}}
\newcommand{\ee}{\end{enumerate}}

\newcommand{\ep}{\varepsilon}

\newcommand{\N}{{\mathbb{N}}}

\newcommand{\U}{{\mathcal{U}}}

\newcommand{\D}{{\mathcal{D}}}

\newcommand{\id}{{\mathrm{id}}}

\newcommand{\Prim}{{\mathrm{Prim}}}

\newcommand{\Aut}{{\mathrm{Aut}}}

\newcommand{\Ad}{{\mathrm{Ad}}}

\newcommand{\ca}{$C^*$-algebra}

\newcommand{\I}{\infty}

\date{\today}
\title[]{Equivariant bundles and absorption}

\thanks{
The first named author was supported by GA\v{C}R project 19-05271Y, RVO:67985840, and by a grant from IPM.
The second named author was partially supported
by the Deutsche Forschungsgemeinschaft through an \emph{eigene Stelle}
and under Germany's Excellence Strategy ``Mathematics 
M\"unster: Dynamics--Geometry--Structure'', and also by a Postdoctoral Research Fellowship
from the Humboldt Foundation.} 

\author{Marzieh Forough}
\address{Marzieh Forough
Institute of Mathematics, Czech Academy of Sciences
115 67
Praha 1, Czech republic}
\email{forough@math.cas.cz}
\urladdr{https://users.math.cas.cz/~forough/}

\author[Eusebio Gardella]{Eusebio Gardella}
\address{Eusebio Gardella
Mathematisches Institut, Fachbereich Mathematik und Informatik der
Universit\"at M\"unster, Einsteinstrasse 62, 48149 M\"unster, Germany.}
\email{gardella@uni-muenster.de}
\urladdr{www.math.uni-muenster.de/u/gardella/}

\begin{document}

\begin{abstract}
For a locally compact group $G$ and a strongly 
self-absorbing $G$-algebra $(\mathcal{D},\delta)$, we
obtain a new characterization of absorption of a 
strongly self-absorbing action using almost equivariant completely positive maps into the underlying algebra. 
The main technical tool to obtain this characterization is the existence of almost equivariant lifts for equivariant completely positive maps, proved in recent work of the authors with Thomsen. 

This characterization is then used to show 
that an equivariant $C_0(X)$-algebra with 
$\mathrm{dim}_{\mathrm{cov}}(X)<\infty$ is $(\mathcal{D},\delta)$-stable
if and only if all of its fibers are, extending a result
of Hirshberg, R\o rdam and Winter to the equivariant setting. The condition on
the dimension of $X$ is known to be necessary, and we 
show that it can be removed if, for example, the bundle is 
locally trivial. 
\end{abstract}

\maketitle

\renewcommand*{\thetheoremintro}{\Alph{theoremintro}}
\section{Introduction}

The study of 
$C_0(X)$-algebras has attracted a great deal of attention,
as their global structure is closely related to
that of the fibers. Indeed, there are many
results which state that, under varying sets of assumptions, a $C_0(X)$-algebra satisfies a given property
whenever all the fibers do; for example, see~\cite{Blaroro_properly_2004, Dadar_kk-fiberwise_2009, GarHirSan_rokhlin_2019}, etc. (The converse is usually also
true and often easier to prove.) 
This work is motivated by a preservation result by
Hirshberg, R\o rdam and Winter \cite{HirRorWin_algebras_2007}, stating that for a 
given strongly self-absorbing \ca\ $\D$, 
a $C_0(X)$-algebra with $\dim_{\mathrm{cov}}(X)<\I$ 
absorbs $\D$ if and only if each fiber does.

We work in the equivariant setting, and for a locally
compact group $G$ we consider $G$-$C_0(X)$-algebras,
that is, $C_0(X)$-algebras endowed with a fiber-wise 
action of $G$. For example, whenever $A$ is a 
$G$-algebra for which $\Prim(A)/G$ is 
Hausdorff, then $A$ can be naturally described in these
terms. Equivariant bundles arise naturally in practice,
even if one is only interested in actions on \emph{simple}
C*-algebras. For example, if a compact group $G$ acts on
a unital \ca\ $A$ with finite Rokhlin dimension
with commuting towers (see \cite{Gar_rokhlin_2015}),
then the continuous part of $A_\I\cap A'$ is a
$G$-$C(X)$-algebra for a suitable compact space $X$. 
A description of the space $X$ in
this context is in general complicated; see, for 
example, the comments after 
Definition~2.2 in~\cite{GarHirSan_rokhlin_2019}.

Recall (see \cite{Sza_stronglyI_2018,Sza_stronglyII_2018}) that a separable, unital 
$G$-algebra $(\D,\delta)$ is said to be \emph{strongly self-absorbing} if the first factor embedding $\D \to \D \otimes \D$ is approximately $G$-unitarily equivalent to an equivariant isomorphism. This notion was introduced by 
Szabo in order to systematize the study of absorption
properties in the dynamical setting, as well as
laying the foundations for the study of internal structures
of dynamical systems; see for example \cite{Sza_rokhlin_2019}.

The main result of this work deals with
equivariant $C_0(X)$-algebras whose fibers absorb a fixed strongly 
self-absorbing $G$-algebra. Dynamical systems of this form appear naturally
when studying the structure of actions on (possibly 
simple) $\mathcal{Z}$-stable \ca s; see
\cite{GarLup_applications_2019, GarHirSan_rokhlin_2019, Sza_rokhlin_2019}. 
Particular cases of the following theorem 
have been proved in these works 
using ad-hoc methods, and our result
gives the desired outcome in greater generality. 

\begin{thmx}\label{thmx:FiberstoBundle}
Let $G$ be a second countable, locally compact group, 
let $X$ be a locally compact Hausdorff space, let 
$(A, \alpha)$ be a separable, unital $G$-$C_0(X)$-algebra, and 
$(\D,\delta)$ be a unitarily regular 
strongly self-absorbing $G$-algebra.
If $X$ is finite-dimensional, then
$(A, \alpha)$ is $(\D, \delta)$-stable if and only if 
all of its fibers are $(\D, \delta)$-stable.
The condition on $\dim_{\mathrm{cov}}(X)$ is necessary,
and can 
be removed if the bundle is locally $(\D,\delta)$-stable (in particular, if the bundle is locally trivial). 
\end{thmx}

Unitary regularity (see \autoref{df:UnitReg}) 
is a very mild condition which in practice serves as an equivariant 
analog of $K_1$-injectivity. Unlike $K_1$-injectivity,
however, it is not known whether unitary regularity
is automatic for strongly self-absorbing actions, 
although this 
has been confirmed when $G$ is amenable; see \cite{GarHir_strongly_2018}.
Moreover, 
since any $(\mathcal{Z},\id_{\mathcal{Z}})$-stable $G$-algebra is automatically
unitarily regular, Theorem~\ref{thmx:FiberstoBundle} is applicable to a wide
family of actions. 

The proof of Theorem~\ref{thmx:FiberstoBundle} roughly
follows the arguments in \cite{HirRorWin_algebras_2007},
adapted to our setting.
Since we work with $G$-algebras, we have to keep track
of equivariance conditions at all steps of the construction. A second and technically
much more challenging task is that of keeping
track of continuity conditions for the action. Indeed, 
whenever one wishes to apply a reindexation argument in the 
(central) sequence algebra, or a diagonal argument, the
topology of the acting group may lead to problems. 
This is because the most direct ways of lifting relations in
$A_{\I}$ or $F(A)$ will only keep track of 
finitely many elements of the group, which is 
in general not enough for most purposes.

The starting point of our proof for Theorem~\ref{thmx:FiberstoBundle} is a new 
characterization of $(\D,\delta)$-absorption in a local
manner. More 
explicitly, our characterization is stated in terms
of almost equivariant completely positive contractive 
maps into the coefficient algebra; see \autoref{thm:CharactDabs}. We reproduce part of the 
statement in the unital case:

\begin{thmx}\label{thm-local-intro}
Let $G$ be a second countable, locally compact group, let 
$(A,\alpha)$ be a separable unital $G$-algebra, and let $(\D,\delta)$
be a strongly self-absorbing $G$-algebra. 
Then the following are equivalent:
\begin{itemize}
\item[(a)] $(A, \alpha)$ is $(\D, \delta)$-stable.
\item[(b)] Given finite subsets $F_\D\subseteq \D$ and
$F_A\subseteq A$, a compact 
subset $K_G\subseteq G$, and $\ep>0$, 
there is a completely positive contractive map 
$\psi \colon \D\to A$ such that
\begin{itemize}
\item[(U)] $\|\psi(1)-1\| < \ep$;
\item[(C)] $\|a\psi(d)-\psi(d)a\| < \ep$;
\item[(M)] $\|\psi(dd')-\psi(d)\psi(d')\| < \ep$;
\item[(E)] $\|\alpha_g(\psi(d))-\psi(\delta_g(d))\|<\ep$;
\end{itemize}
for all $a \in F_A$ and $d, d' \in F_\D$ and all $g\in K_G$.
\item[(c)] Given finite subsets $F_\D\subseteq \D$ and
$F_A\subseteq A$, a compact subset $K_G\subseteq G$, and $\ep>0$, 
there are unital $\ast$-homomorphisms
$\mu \colon \D\to A$ and $\sigma \colon A\to A$ with commuting 
ranges that satisfy the following 
for $a\in F_A$, $d\in F_\D$, and $g\in K_G$:
\bi
\item[\ (I) \ ] $\|\sigma(a)-a\|<\ep$;
\item[\ (E)$_\mu$] $\|\alpha_g(\mu(d))-\mu(\delta_g(d))\|<\ep$;
\item[\ (E)$_\sigma$] $\|\alpha_g(\sigma(a))-\sigma(\alpha_g(a))\|<\ep$;
\ei
\end{itemize} 
\end{thmx} 

In condition (b) above, it is possible to work with \emph{finite} subsets of $G$
instead, by adding a condition controling the continuity modulus of $\alpha$
on the range of $\psi$; see part~(4) of \autoref{thm:CharactDabs}.

Throughout this work, and particularly in the proof 
of Theorem~\ref{thm-local-intro}, the existence 
of almost equivariant lifts for 
equivariant completely positive maps, obtained in 
\cite{ForGarTom_asymptotic_2021}, is a crucial tool.
For compact, or even amenable groups, said lifting 
results are easier to obtain and have been used in the 
past to prove particular cases of Theorem~\ref{thmx:FiberstoBundle}. Our biggest contribution
is thus in the non-amenable setting, where averaging 
arguments cannot be carried out and one instead has to 
resort to more delicate arguments; see specifically
Section~2 in~\cite{ForGarTom_asymptotic_2021}. 
\vspace{.2cm}

\textbf{Acknowledgement.} The present work is
the second part of a project initiated 
together with Klaus Thomsen in
\cite{ForGarTom_asymptotic_2021}, and 
Klaus was initially a 
coauthor in this paper. After posting it 
to the arxiv, he decided he had not contributed 
enough, and preferred not to be a coauthor. 
This paper benefited greatly from his insights, 
for which we are very thankful.

\section{Absorption of strongly self-absorbing actions}
In this section, we recall the definition and some basic facts about 
strongly self-absorbing actions, and obtain a new characterization of 
absorption of such an action in terms of almost equivariant and almost
multiplicative completely positive maps into the given algebra; see \autoref{thm:CharactDabs}.

Given a locally compact group $G$, a \emph{$G$-algebra} is a pair $(A,\alpha)$ consisting of a \ca\ $A$ and a group homomorphism 
$\alpha\colon G\to\Aut(A)$, also called an \emph{action},
which is continuous with respect to the 
point-norm topology in $\Aut(A)$.

We begin by reviewing some results obtained in
\cite{ForGarTom_asymptotic_2021} together with Thomsen, about lifts of
equivariant completely positive contractive maps.

\subsection{Lifts of completely positive equivariant maps}
Let $(A,\alpha)$ and $(B,\beta)$ be $G$-algebras, let 
$I$ be a $G$-invariant
ideal in $B$ with associated quotient map $\pi\colon B\to B/I$, and let $\varphi \colon A\to B/I$ be an equivariant completely positive contractive map. 
\[\xymatrix{& B\ar[d]^{\pi}\\
A\ar[r]_-{\varphi}\ar@{-->}^{\psi}[ur] & B/I}\]
If $A$ is nuclear (as it will be in the situations we are interested in), then the Choi-Effros lifting theorem 
\cite{ChoEff_completely_1976} guarantees the existence 
of a completely positive contractive map $\psi\colon A\to B$
making the above diagram commutative.
Since we work in the equivariant category, we will 
be interested in obtaining lifts for $\varphi$ which
are ``almost'' equivariant. We will
control the failure of equivariance of a lift $\psi$ 
by obtaining small bounds for the quantity
\[\max_{a\in F_A}\max_{g\in K_G}\|(\psi\circ\alpha_g)(a)-(\beta_g\circ\psi)(a)\|,\] 
for a finite subset $F_A\subseteq A$ and a compact
subset $K_G\subseteq G$. We call the above quantity the 
\emph{equivariance modulus} of $\psi$ with respect to 
$F_A$ and $K_G$.

In fact, we will need to work in a more general setting,
and we will in particular not assume that $\psi$ is 
equivariant. Hence, we do not expect to find almost
equivariant lifts, and instead we will
be interested in finding lifts whose equivariance
moduli are controled by those of $\psi$. This was achieved jointly with Thomsen
in \cite{ForGarTom_asymptotic_2021}, and 
we recall the precise statement for use in this work.

\begin{thm}\label{thm:EqChoiEffros}
Let $G$ be a second countable, locally compact group, let $(A,\alpha)$ 
and $(B,\beta)$ be $G$-algebras with $A$ separable. Let $I$ be a 
$\beta$-invariant 
ideal in $B$. Denote by $\pi\colon B\to B/I$ the canonical quotient map, 
and by $\overline{\beta}\colon G\to \Aut(B/I)$
the induced action. Let $\varphi\colon A\to B/I$ be a
nuclear completely 
positive contractive map.

Given $\varepsilon>0$, a finite subset $F_A\subseteq A$ and a compact subset $K_G\subseteq G$, 
there exists a completely 
positive contractive map $\psi\colon A\to B$ with $\pi\circ \psi=\varphi$,
satisfying
\[\|(\psi\circ\alpha_g)(a)- (\beta_g\circ\psi)(a)\| \leq  
\|(\varphi\circ\alpha_g)(a)- (\overline{\beta}_g\circ\varphi)(a)\| +\ep\]
for all $g\in K_G$ and $a\in F_A$.
\end{thm}
\begin{proof} This follows immediately from combining
the Choi-Effros lifting theorem \cite{ChoEff_completely_1976}
with Theorem~3.4 in~\cite{ForGarTom_asymptotic_2021}.
\end{proof}

When $A$, $B$ and $\varphi$ are unital, one cannot
always construct lifts as above which are moreover unital;
see Section~4 in~\cite{ForGarTom_asymptotic_2021}.
In our setting, a unital, asymptotically equivariant 
lifts exist if and only if 
there is a sequence of asymptotically
$G$-equivariant completely positive unital maps 
$A\to B$, which is in particular the case if $B$
admits a $G$-invariant state.

\subsection{Strongly self-absorbing actions}
For a unital C*-algebra $D$, we write $\mathcal{U}(D)$
for its unitary group.

Let $(A,\alpha)$ and $(B,\beta)$ be $G$-algebras, 
and let $\widetilde{\beta}$ denote the strictly
continuous extension of $\beta$ to $M(B)$. 
Recall that two 
equivariant homomorphisms $\varphi,\psi\colon (A,\alpha)\to (B,\beta)$ are said to be 
\emph{approximately $G$-unitarily equivalent}, written
$\varphi\approx_{G,u} \psi$, if for every $\ep>0$, every compact subset
$K\subseteq G$, and all finite subsets
$F_A\subseteq A$ and $F_B\subseteq B$, 
there exists a unitary $u\in M(B)$ such that $\|\varphi(a)-u\psi(a)u^*\|<\ep$
for all $a\in F_A$ and 
\[\max_{b\in F_B}\max_{g\in K}\|\widetilde{\beta}_g(u)b-ub\|<\ep  \ \ \mbox{ and } \ \ 
\max_{b\in F_B}\max_{g\in K}\|b\widetilde{\beta}_g(u)-bu\|<\ep.\]
Equivalently, there exists a net $(u_j)_{j\in J}$ in $\mathcal{U}(M(B))$
such that $\widetilde{\beta}_g(u_j) - u_j$ 
converges strictly to zero in $M(B)$, 
uniformly on compact subsets
of $G$, and $\Ad(u_j)\circ \psi$ converges pointwise to $\varphi$ in norm. 

\begin{df} \label{df:ssa}
Let $G$ be a second countable, locally compact group and let
$(\D,\delta)$ be a $G$-algebra. We say that $(\D,\delta)$ is 
\emph{strongly self-absorbing}, if $\D$ 
is separable, unital and infinite-dimensional, and there
is an equivariant isomorphism $(\D,\delta)\cong (\D\otimes\D,\delta\otimes\delta)$ 
which is approximately $G$-unitarily equivalent to the (equivariant) first factor embedding 
$\mathrm{id}_{\D} \otimes 1 \colon (\D, \delta)\to (\D \otimes \D, \delta \otimes \delta)$.
\end{df}

If $(\D,\delta)$ is a strongly self-absorbing $G$-algebra, then clearly $\D$
is strongly self-absorbing in the sense of \cite{TomWin_strongly_2007}. In particular, 
$\D$ is simple and nuclear; 
see Proposition~1.5 and Theorem~1.6 in~\cite{TomWin_strongly_2007}.

The property of being strongly self-absorbing is not a common one, and most of the 
naturally occurring examples of ($G$-)algebras fail to be 
strongly self-absorbing. On the other hand, the property of \emph{absorbing} a given
strongly self-absorbing $G$-algebra is a much more common and 
useful one.
Here, absorption is always meant up to cocycle conjugacy, in the following sense:

\begin{df}\label{df:cocycleEq}
Let $G$ be a locally compact group and let 
$(A,\alpha)$ and $(B,\beta)$ be $G$-algebras, 
and let $\omega \colon G\to \U(M(A))$ be a strictly continuous map.
\begin{itemize}
\item[(a)] We say that $\omega$ is a \emph{1-cocycle for $\alpha$} if 
$\omega_g\alpha_g(\omega_{h})=\omega_{gh}$ for all $g,h\in G$. In this case, 
we let $\alpha ^{\omega}\colon G\to\Aut(A)$ be the action 
given by $\alpha^{\omega}_g=\Ad(\omega_g) \circ \alpha_g$ for all $g\in G$.
\item[(b)] We say that $\alpha$ and $\beta$ are \emph{cocycle conjugate}, 
denoted $\alpha \simeq_{cc} \beta$, if there exists a 1-cocycle $\omega$ 
for $\alpha$ such that $\alpha^\omega$ is conjugate to $\beta$.
\end{itemize}

If $(D,\delta)$ is another $G$-algebra, we say that $(A,\alpha)$ \emph{absorbs
$(D,\delta)$}, or that $(A,\alpha)$ is \emph{$(D,\delta)$-stable}, 
if $\alpha\otimes\delta\simeq_{cc} \alpha$.
\end{df}

When $(\D,\delta)$ is strongly self-absorbing, $\delta$-absorption
can be characterized in terms of central sequence algebras; see 
Theorem~3.7 of~\cite{Sza_stronglyI_2018}. 
In this section, we use \autoref{thm:EqChoiEffros} to obtain a new characterization of 
this property in terms of completely positive maps into the given algebra 
(instead of its central sequence); see \autoref{thm:CharactDabs}.
Our result is an equivariant version of Theorem~4.1 
in~\cite{HirRorWin_algebras_2007}, but we warn the reader that there
are several additional technicalities in the equivariant setting,
mostly related to the lack of continuity of the induced action on 
the (central) sequence algebra.  

\begin{df}\label{df:CentralSeqAlg}
Let $A$ be a \ca. We write $A_\I$ for its 
\emph{sequence algebra}, that is, $A_\I=\ell^\I(\N,A)/c_0(\N,A)$, 
and we write $\pi_A\colon \ell^\I(\N,A)\to A_\I$ for the canonical
quotient map.
Identifying $A$ with the subalgebra of $\ell^\I(\N,A)$ consisting of the constant
sequences, and with its image in $A_\I$, we write $A_\I\cap A'$ for the relative 
commutant, which we call the \emph{central sequence algebra}. We let
\[\mathrm{Ann}(A,A_\I)=\{x\in A_\I\colon xA=Ax=\{0\}\}\]
denote the \emph{annihilator} of $A$ in $A_\I$. Then $\mathrm{Ann}(A,A_\I)$ is 
an ideal in $A_\I\cap A'$ (but not in general an ideal in 
$A_\I$). We let 
$F_\infty(A)$ denote the corresponding quotient, and write 
$\kappa_A\colon A_\I\cap A'\to F_\infty(A)$ for the canonical quotient map. 

Assume that $A$ is $\sigma$-unital. 
If $(a_n)_{n\in\N}$ is an approximate unit for $A$, and denoting by 
$a\in \ell^\I(\N,A)$ the element it determines, then 
$\kappa_A(\pi_A(a))$ is a (and hence \emph{the}) unit for $F_\infty(A)$.
\end{df} 




Let $G$ be a locally compact group, 
and let $\alpha\colon G\to\Aut(A)$ be a continuous action. For each $g\in G$, the 
automorphism $\alpha_g$ induces canonical automorphisms of 
$\ell^\I(\N,A)$, $A_\I\cap A'$ and $F_\infty(A)$, which will be denoted 
by $\alpha_g^\I$, $(\alpha_\I)_g$ and $F_\infty(\alpha)_g$, respectively. 
The resulting maps $\alpha^\I$, $\alpha_\I$ and $F(\alpha)$ are group 
homomorphisms and make the quotient maps $\pi_A$ and $\kappa_A$ equivariant,
but they fail in general
to be \emph{continuous} actions. 
In order to lighten the notation, we will denote the continuous
parts of $\ell^\I(\N,A)$, $A_\I$ and $F_\I(A)$ by
$\ell^\I_\alpha(\N,A)$, $A_{\I,\alpha}$ and $F_{\I,\alpha}(A)$,
respectively. Observe that the continuous part of $A_\I\cap A'$
is just $A_{\I,\alpha}\cap A'$.

Since the lifting results presented in the previous subsection apply only to 
continuous actions and surjective maps, 
the following will be needed in the sequel.

\begin{prop}\label{prop:CtsPartQuotient}
Let $G$ be a locally compact group and let $(A,\alpha)$ be a 
$G$-algebra.
Then $\pi_A\left( \ell^\I(\N,A)\right) =A_{\I,\alpha}$, and hence 
the restriction $\pi_A\colon \ell^\I_\alpha(\N,A)\to 
A_{\I,\alpha}$
of $\pi_A$ to the continuous part of 
$\ell^\I(\N,A)$ is an equivariant quotient 
map. If $A$ is separable, then
$\kappa_A(A_{\I,\alpha}\cap A') =  F_{\I,\alpha}(A)$ and thus 
$\kappa_A\colon A_{\I,\alpha}\cap A'\to F_{\I,\alpha}(A)$ is also an equivariant quotient map. 
\end{prop} 
\begin{proof}
The result for $\pi_A$ is a consequence of 
Theorem~2 in~\cite{Bro_continuity_2000},
since it follows from said theorem
that the preimage under $\pi_A$ of any 
element in $A_{\I,\alpha}$ already
belongs to $\ell^\I_\alpha(\N,A)$.
When $A$ is separable, the claim about 
$\kappa_A$ follows from 
Lemma~4.7 in~\cite{Sza_stronglyII_2018}.
\end{proof}

\subsection{Local characterizations of absorption}
In \autoref{thm:CharactDabs}, we give new
characterizations 
for a $G$-algebra $(A,\alpha)$ to absorb a given
strongly self-absorbing $G$-algebra 
in terms of 
almost equivariant maps into $A$. 
The following notation will be used in its statement and proof, 
and will also be used repeatedly in the following section.

\begin{nota}\label{nota:Abbreviations}
Let $G$ be a locally compact group, let $(A,\alpha)$ and $(D,\delta)$
be $G$-algebras with $D$ unital, and let $\psi\colon D\to A$ be a completely positive 
contractive map. Given finite subsets $F_D\subseteq D$ and $F_A\subseteq A$,
a compact subset 
$K_G\subseteq G$, and $\ep>0$, we 
say that $\psi$ is
\bi
\item[U:] $(F_A,\ep)$-\underline{u}nital, if $\|\psi(1)a-a\|<\ep$ for all $a\in F_A$; 
\item[C:] $(F_D,F_A,\ep)$-\underline{c}entral, if $\|a\psi(d)-\psi(d)a\|<\ep$ for all $a\in F_A$
and $d\in F_D$;
\item[M:] $(F_D,F_A,\ep)$-\underline{m}ultiplicative, if $\|a(\psi(dd')-\psi(d)\psi(d'))\|<\ep$
for all $a\in F_A$ and all $d,d'\in F_D$;
\item[E:] $(F_D,K_G,\ep)$-\underline{e}quivariant, if $\max\limits_{g\in K_G}\|\alpha_g(\psi(d))-\psi(\delta_g(d))\|<\ep$
for all $d\in F_D$. 
\ei 
Whenever the tuple $(F_D,F_A,K_G,\ep)$ is clear from the context, 
we will refer to the above conditons as 
(U), (C), (M), and (E).
\end{nota}

In the proof of the following theorem, 
we will repeatedly apply the lifting results from Subsection~2.1 to the qutient maps 
$\pi_A$ and $\kappa_A$. When lifting
along $\pi_A$, we will need the most general lifting result from
\autoref{thm:EqChoiEffros}, where the given map is not 
assumed to be equivariant. 

We make some comments about part~(3) of
\autoref{thm:CharactDabs} below. 
It doesn't seem possible to give a ``finitistic'' characterization 
(that is, in terms of finite sets and $\ep>0$) of the existence of a map
into the \emph{continuous part} of $A_\I$ (or of $F(A)$). Indeed,
for a map into $A_\I$ to land in the continuous part, the individual 
maps into $A$ ought to be \emph{equicontinuous}, and this cannot
be described in a finitistic manner. The solution that we adopted takes 
advantage of the fact that the resulting map into $A_\I$ will have 
the property of being equivariant, and thus the continuity moduli
of the codomain action are controlled by the continuity moduli
of the domain action. 

\begin{thm}\label{thm:CharactDabs}
Let $G$ be a second countable, locally compact group and 
let $(A,\alpha)$ be a separable $G$-algebra. 
For a strongly self-absorbing $G$-algebra
$(\D,\delta)$, consider the following conditions:
\be
\item $(A, \alpha)$ is $(\D,\delta)$-stable.

\item There exists a 
completely positive contractive map 
$\sigma\colon \D\to \ell^\I(\N,A)$ such that the range 
of $\pi_A\circ\sigma$ is contained in 
$A_{\infty, \alpha} \cap A'$ and $\pi_A\circ\sigma\colon 
\D\to A_{\infty, \alpha}$ is equivariant,
and such 
that $\kappa_A\circ\pi_A\circ\sigma$ is a 
unital homomorphism.

\item For all finite sets $F_\D \subseteq \D$, and $F_A\subseteq A$, all compact
subsets $K_G\subseteq G$, and $\ep>0$,
there is a completely positive contractive map 
$\psi \colon \D\to A$ 
which satisfies conditions (U), (C), (M),
and (E)
for $(F_\D,F_A,K_G,\ep)$.
\item For all finite sets $F_\D \subseteq \D$, $F_A\subseteq A$, and 
$F_G\subseteq G$, every compact neighborhood of the unit $N\subseteq G$,
and every $\ep>0$, there is a completely positive contractive map 
$\psi \colon \D\to A$ 
which satisfies conditions (U), (C), (M),
and (E)
for $(F_\D,F_A,F_G,\ep)$ as well as 
\[\max\limits_{h\in N}\|\alpha_h(\psi(d))-\psi(d)\|\leq 
\max\limits_{h\in N}\|\delta_h(d)-d\|+\ep\]
for all $d\in F_D$.
\item Given finite subsets $F_\D\subseteq \D$, and $F_A \subseteq A$, a compact
subset $F_G \subseteq G$, and $\ep>0$, 
there exist homomorphisms
$\mu\colon \D \to M(A)$ and $\sigma\colon A\to A$
with $\mu(1)=1$, satisfying:
\be
\item[(5.a)] $\sigma(a)\mu(d)=\mu(d)\sigma(a)$ for all $d\in\D$
and all $a\in A$;
\item[(5.b)] $\|\sigma(a)-a\|<\ep$ for all $a\in F_A$;
\item[(5.c)] $\max\limits_{g\in F_G}\|\mu(\delta_g(d))\alpha_g(a)-\alpha_g(\mu(d)a)\|<\ep\|a\|$ for all $a\in A$ and $d\in F_\D$; and 
\item[(5.d)] $\max\limits_{g\in F_G}\|\sigma(\alpha_g(a))-\alpha_g(\sigma(a))\|<\ep$ for all $a\in F_A$.
\ee
\ee
Then (1), (2), (3) and (4) are equivalent, and (5) implies all 
of them. 
If $A$ is unital then they are all equivalent.
\end{thm}
\begin{proof}
(1) implies (2): let $(F_\D^{(n)})_{n\in\N}$
and $(F_A^{(n)})_{n\in\N}$
be increasing sequences of finite subsets 
of $\D$ and $A$, respectively, with dense union, and such 
that $\bigcup_{n\in\N}F_\D^{(n)}$ is self-adjoint and closed
under scalar multiplication by elements from $\mathbb{Q}[i]$. Let $(K_G^{(n)})_{n\in\N}$
be an increasing sequence of compact subsets
of $G$ with $G=\bigcup_{n\in\N} K_G^{(n)}$.

Since $(A,\alpha)$ is $(\D,\delta)$-absorbing, we may use Theorem~3.7 
in~\cite{Sza_stronglyI_2018} to fix a unital, equivariant homomorphism 
$\Psi\colon (\D,\delta)\to (F_{\infty,\alpha}(A),F_\infty(\alpha))$. 
Then $\Psi$ is nuclear, because so is $\D$.
By \autoref{thm:EqChoiEffros} and \autoref{prop:CtsPartQuotient}, 
for every $n\in\N$ there exists a
completely positive contractive map 
$\psi_n\colon \D\to A_{\I,\alpha}\cap A'$ satisfying 
$\kappa_A\circ \psi_n=\Psi$ and 
\[\sup_{g\in K_G^{(n)}}\|\psi_n(\delta_g(d))-(\alpha_\I)_g(\psi_n(d))\|<\frac{1}{n}\]
for all $d\in F_\D^{(n)}$. 
Use \autoref{thm:EqChoiEffros}
to find a completely positive contractive map 
$\sigma_n\colon \D\to \ell^\I(\N,A)$ 
with 
\be
\item[(i)] $\pi_A\circ\sigma_n=\psi_n$ and 
\item[(ii)] $\max_{g\in K_G^{(n)}}\|\sigma_n(\delta_g(d))-\alpha^\I_g(\sigma_n(d))\|<1/n$
for all $d\in F_\D^{(n)}$.
\ee
Fix $n\in\N$. Since the image of $\psi_n$ commutes 
with $A$, we have
\[\lim_{k\to\I}\|\sigma_n(d)_ka-a\sigma_n(d)_k\|=0\]
for all $d\in \mathcal{D}$ and all $a\in A$. In particular, there exists
$m_n\in\N$ such that for all $m\geq m_n$ we have
$\|\sigma_n(d)_ma-a\sigma_n(d)_m\|<1/n$ for all $d\in F_\D^{(n)}$ and all 
$a\in F_A^{(n)}$. Thus, upon redefining the first 
$m_n-1$ entries of $\sigma_n$ to be zero, we may 
also assume that
\be
\item[(iii)] $\|\sigma_n(d)_ka-a\sigma_n(d)_k\|<1/n$ for all $k\in\N$, all $d\in F_\D^{(n)}$ and all $a\in F_A^{(n)}$.\ee
Fix $n\in\N$. Since $\sigma_n$ and $\sigma_1$ are both 
lifts of $\Psi$, there exists $k_n\in\N$ such that
whenever $k\geq k_n$, we have
\be
\item[(iv)] 
$\|a(\sigma_n(d)_k-\sigma_1(d)_k)\|<1/n$ for all 
$a\in F_A^{(n)}$ and all $d\in F_\D^{(n)}$.
\ee
Without loss of generality, we may assume that 
$(k_n)_{n\in\N}$ is a strictly increasing sequence.
Define a completely positive
contractive map $\sigma\colon \D\to \ell^\I(\N,A)$ 
by setting
\[\sigma(d)_k=\begin{cases}
                 0, & \mbox{if } k<k_1,\\
                 \sigma_1(d)_k, & \mbox{if } k_1\leq k< k_2;\\
                 \vdots &\\
                 \sigma_n(d)_k, & \mbox{if } k_n\leq k<k_{n+1};\\
                 \vdots
                 \end{cases}\]
for all $d\in\D$ and all $k\in\N$. 
Set $\psi=\pi_A\circ \sigma\colon \D\to A_{\I}$. 
Then $\psi$ is completely positive and contractive as well, 
and its 
range is contained in $A_{\I}\cap A'$ by condition (iii) above.
Condition (ii) implies that $\psi$ is equivariant, and thus its 
range is contained in the continuous part $A_{\I,\alpha}\cap A'$. 

We claim that 
$\kappa_A\circ\psi=\kappa_A\circ\pi_A\circ\sigma_1$. 
By continuity, given $m\in\N$, $a\in F_A^{(m)}$ and 
$d\in F_\D^{(m)}$, it suffices to check that 
\[\lim_{k\to\I} \|a(\sigma_1(d)_k-\sigma(d)_k)\|=0.\]
Let $\ep>0$, and find $n\geq m$ with $1/n<\ep$. 
Given $k\geq k_n$, let $r\in\N$ satisfy $k_r\leq k< k_{r+1}$. (Note that $r\geq n$.) 
Using that $a\in F_A^{(r)}$
and $d\in F_\D^{(r)}$, we have
\[a\sigma_1(d)_k\stackrel{\mathrm{(iv)}}{\approx}_{\frac{1}{r}} a\sigma_r(d)_k = a\sigma(d)_k.\]
We deduce that $\|a(\sigma_1(d)_k-\sigma(d)_k)\|<\ep$ for all $
k\geq k_n$. Since $\ep$ is arbitrary, this proves the 
claim.

We deduce that 
$\kappa_A\circ \psi=\kappa_A\circ\pi_A\circ\sigma_1=\Psi$, so $\kappa_A\circ\psi$
is a unital 
homomorphism. We have proved (2).

(2) implies (3): let $\ep>0$, let $F_A\subseteq A$ and $F_\D\subseteq \D$ 
be finite subsets, and let $K_G\subseteq G$ be 
a compact subset. 
Without loss of generality, we assume that 
$F_A$ consists of contractions. 
Let $\sigma\colon \D\to\ell^\I(\N,A)$ be a map as in the 
statement, and set $\varphi=\pi_A\circ\sigma$.
Since $\D$ is nuclear, the map $\varphi\colon \D\to A_{\I,\alpha}\cap A'$ is also nuclear. 
By \autoref{thm:EqChoiEffros} and
\autoref{prop:CtsPartQuotient}, 
there exists a completely
positive contractive map $\Psi\colon \D\to \ell_\alpha^\I(\N,A)$
which is $(F_\D,K_G,\ep)$-equivariant, and satisfies $\pi_A\circ \Psi=\varphi$. For $n\in\N$, let $\psi_n\colon \D\to A$ denote the 
composition of $\Psi$ with the evaluation map at $n\in\N$. 
By definition of the norm on $F_\infty(A)$, we have
\bi
\item $\lim_{n\to\I}a\psi_n(1)=a$ for all $a\in A$;
\item $\lim_{n\to\I}\|a\psi_n(d)-\psi_n(d)a\|=0$ for all $a\in A$ and all $d\in \D$;
\item $\lim_{n\to\I}\|a(\psi_n(dd')-\psi_n(d)\psi_n(d'))\|=0$ for all $a\in A$ and all $d,d'\in \D$.
\ei
Additionally, since $\Psi$ is $(F_\D,K_G,\ep)$-equivariant,
we also have
\bi
\item $\limsup_{n\to\I}\|\psi_n(\delta_g(d))-\alpha_g(\psi_n(d))\|<\ep$ for $g\in K_G$ and $d\in F_\D$.
\ei
The proof now follows by setting $\psi=\psi_n$ for $n$ large enough.
(3) implies (4): this is immediate, since any map satisfying (U), (C), (M) 
and (E) for $(F_\D, F_A, F_G\cup N, \ep)$ also satisfies the displayed
inequality in (4).

(4) implies (1): let $(F_\D^{(n)})_{n\in\N}$ and
$(F_A^{(n)})_{n\in\N}$ be increasing sequences of finite subsets 
of $\D$ and $A$, respectively, with dense union. Using that $G$ is
locally compact and second countable, let 
$(F_G^{(n)})_{n\in\N}$ be an increasing sequence of finite subsets
of $G$ whose union $G^{(0)}$ is a dense subgroup of $G$, and let 
$(N^{(n)})_{n\in\N}$ be a decreasing sequence of compact 
neighborhoods of $1_G$ which form a basis of neighborhoods
for $1_G$.
For each $n\in\N$, let 
$\psi_n\colon \D\to A$ be a 
completely positive contractive map 
which satisfies condition (U), (C), (M) and (E) 
from \autoref{nota:Abbreviations} with respect to $(F_\D^{(n)}, 
F_A^{(n)}, F_G^{(n)},1/n)$, and such that
\[\label{eqn:Continuity}\tag{3.6}
\max\limits_{h\in N^{(n)}}\|\alpha_h(\psi_n(d))-\psi_n(d)\|\leq 
\max\limits_{h\in N^{(n)}}\|\delta_h(d)-d\|+\frac{1}{n}\]
for all $d\in F_D^{(n)}$.

Let $\psi\colon \D\to
\ell^\I(\N,A)$ be given by $\psi(d)(n)=\psi_n(d)$ for all $d\in\D$
and all $n\in\N$. 
Then the image of $\psi$ is contained in $A_\I\cap A'$ 
by condition (C) in \autoref{nota:Abbreviations}. 
Set $\Psi=\kappa_A\circ \pi_A\circ\psi$. Then $\Psi$ is a unital
homomorphism by conditions (U) and (M), and it
satisfies 
\begin{equation}\label{eqn:PsiEquiv}\tag{3.7}
\Psi\circ \delta_g=F_\I(\alpha)_g\circ\Psi 
\end{equation}
for all $g\in G^{(0)}$, by condition (E).
On the other hand, for fixed $d\in \D$ the function 
$G\to F_\I(A)$ given by $g\mapsto F_\I(\alpha)_g(\Psi(d))$
is continuous at $1_G$ (and thus at every point of $G$) by
(\ref{eqn:Continuity}). 
It follows that the range of $\Psi$ is contained in
$F_{\I,\alpha}(A)$. By continuity, we deduce that the 
identity in (\ref{eqn:PsiEquiv}) holds also for every 
$g\in \overline{G^{(0)}}=G$.
In other words, $\Psi$ is equivariant.
It then follows from Theorem~3.7 in~\cite{Sza_stronglyI_2018} 
that $(A,\alpha)$ is $(\D,\delta)$-stable, as desired.

It follows that conditions (1), (2), (3) and (4) are equivalent. 

(5) implies (1): Let $(F_\D^{(n)})_{n\in\N}$ and
$(F_A^{(n)})_{n\in\N}$ be increasing sequences of finite subsets 
of $\D$ and $A$, respectively, whose unions are dense in the 
respective unit balls, and let 
$(K_G^{(n)})_{n\in\N}$ be an increasing sequence of compact subsets
of $G$ whose union equals $G$. 
For each $n\in\N$, let 
$\mu_n\colon \D\to M(A)$ and $\sigma_n\colon A\to A$ be 
homomorphisms satisfying 
conditions (5.a) through (5.d) for $F_\D^{(n)}$, $F_A^{(n)}$, 
$K_G^{(n)}$ and $\ep_n=1/n$. 

Use the unnumbered
lemma at the top of page 152 of~\cite{Kas}
to find an approximate unit $(x_n)_{n\in\N}$
for $A$ which satisfies
\bi
\item[(i)] $\|ax_{n}^{1/2}-a\|<1/n$ for all $a\in F_A^{(n)}$;
\item[(ii)] $\|x_{n}^{1/2}\mu_n(d)-\mu_n(d)x_{n}^{1/2}\|<1/n$ for all $d\in F_\D^{(n)}$;
\item[(iii)] $\|\alpha_g(x^{1/2}_{n})-x^{1/2}_{n}\|<1/n$ for all $g\in K_G^{(n)}$.
\ei
By (i), we have
\begin{equation}\tag{iv}
\|ax_{n}^{1/2}-x_{n}^{1/2}a\|<\frac{2}{n} 
\end{equation}
for all $a\in F_A^{(n)}$.
Define a map $\psi\colon \D \to \ell^\I(\N,A)$ by 
\[\psi(d)_n=x_{n}^{1/2}\mu_n(d)x_{n}^{1/2}\] 
for all $d\in \D$ and all $n\in\N$.
We claim that the range of $\pi_A\circ\psi$,
which is a priori a subset of $A_\I$, is contained
in $A_\I\cap A'$. 
Fix $n\in\N$ and let $d\in F_\D^{(n)}$ and let $a\in F_A^{(n)}$. For $m\in\N$ with $m\geq n$, we have 
\begin{align*}
\psi(d)_ma&=x_{m}^{1/2}\mu_m(d)x_{m}^{1/2}a\\
&\stackrel{\mathrm{(iv)}}{\approx_{\frac{2}{m}}} x_{m}^{1/2}\mu_m(d)ax_{m}^{1/2}\\
&\stackrel{(\mathrm{5.b})}{\approx_{\frac{1}{m}}} x_{m}^{1/2}\mu_m(d)\sigma_m(a)x_{m}^{1/2}\\
&\stackrel{(\mathrm{5.a})}{=} x_{m}^{1/2}\sigma_m(a)\mu_m(d)x_{m}^{1/2}\\
&\stackrel{(\mathrm{5.b})}{\approx_{\frac{1}{m}}} x_{m}^{1/2}a\mu_m(d)x_{m}^{1/2}\\
&\stackrel{\mathrm{(iv)}}{\approx_{\frac{1}{m}}} ax_{m}^{1/2}\mu_m(d)x_{m}^{1/2}=a\psi(d)_m.
\end{align*}
Thus $\lim_{m\to\I}\|\psi(d)_ma-a\psi(d)_m\|=0$. 
By density of the linear spans of $\bigcup_{n\in\N} F_\D^{(n)}$ and $\bigcup_{n\in\N}F_A^{(n)}$ in $\D$ and $A$, respectively, 
it follows that the 
range of $\pi_A\circ\psi$ is contained in $A_\I\cap A'$.

Set $\varphi=\kappa_A\circ \pi_A\circ \psi\colon \D\to 
F_\I(A)$. 
Then $\varphi$ is completely positive and contractive, because so is
$\psi$. We claim that $\varphi$ is a unital equivariant homomorphism. 

To check that $\varphi$ is equivariant, fix $n\in\N$ and let 
$d\in F_\D^{(n)}$, $a\in F_A^{(n)}$ and 
$g\in K_G^{(n)}$ be given. For $m\geq n$,
we have
\begin{align*}\psi(\delta_g(d))_ma&=x_{m}^{1/2}\mu_m(\delta_g(d))x_{m}^{1/2}a\\
&\stackrel{(\mathrm{iv})}{\approx_{\frac{2}{m}}}
x_{m}^{1/2}\mu_m(\delta_g(d))ax_{m}^{1/2}\\
&\stackrel{(\mathrm{5.c})}{\approx_{\frac{1}{m}}}  
x_{m}^{1/2}\alpha_g(\mu_m(d)\alpha_{g^{-1}}(a))x_{m}^{1/2}\\
&\stackrel{(\mathrm{iii})}{\approx_{\frac{1}{m}}}  
\alpha_g\big(x_{m}^{1/2}\mu_m(d)\big)ax_{m}^{1/2}\\
&\stackrel{(\mathrm{iv})}{\approx_{\frac{2}{m}}}
\alpha_g\big(x_{m}^{1/2}\mu_m(d)\big)x_{m}^{1/2}a\\
&\stackrel{(\mathrm{iii})}{\approx_{\frac{1}{m}}}
\alpha_g\big(x_{m}^{1/2}\mu_m(d)x_{m}^{1/2}\big)a\\
&=\alpha_g(\psi(d)_m)a.
\end{align*}
We deduce that $\lim_{m\to\I}\|\big(\psi(\delta_g(d))_m-\alpha_g(\psi(d)_m)\big)a\|=0$. By the definition of
$F_\I(A)$, this implies that $\varphi=\kappa_A\circ\pi_A\circ\psi$ is equivariant. 

Recall that whenever $(a_n)_{n\in\N}$ is an approximate unit in 
$A$, regarded as an element $a\in \ell^\I(\N,A)$, then 
$\kappa_A(\pi_A(a))$ is the unit of $F_\infty(A)$. Since 
\[\varphi(1)=(\kappa_A\circ\pi_A\circ\psi)(1)=\kappa_A(\pi_A((a_{n})_{n\in\N})), \]
it follows that $\varphi$ is unital. Finally, to check
that $\varphi$ is a homomorphism, fix $n\in\N$ and let 
$d,e\in F_\D^{(n)}$ and $a\in F^{(n)}_A$ be given. 
For every $m\geq n$, we have 
\begin{align*}
\psi(de)_ma&=x_{m}^{1/2}\mu_m(de)x_{m}^{1/2}a\\
&= x_{m}^{1/2}\mu_m(d)\mu_m(e)x_{m}^{1/2}a\\
&\stackrel{(\mathrm{i})}{\approx_{\frac{2}{m}}} 
x_{m}^{1/2}\mu_m(d)\mu_m(e)x_{m}^{3/2}a\\
&\stackrel{(\mathrm{ii})}{\approx_{\frac{2}{m}}} 
x_{m}^{1/2}\mu_m(d)x_{m}\mu_m(e)x_{m}^{1/2}a\\
&=\psi(d)_m\psi(e)_ma.
\end{align*}
Thus $\lim_{m\to\I}\|\psi(de)_ma-\psi(d)_m\psi(e)_ma\|=0$.
By definition of $F_\infty(A)$, this shows that 
$\varphi=\kappa_A\circ\pi_A\circ \psi$ is a homomorphism. 

Since $\varphi$ is equivariant, its image
is contained in $F_{\I,\alpha}(A)$. Thus the equivalence
between (ii) and (iii) in 
Theorem~3.7 in~\cite{Sza_stronglyI_2018} implies that
$(A,\alpha)$ is $(\D,\delta)$-stable, as desired.

Assume now that $A$ is unital; we will show that
(1) implies (5).
Let $\ep>0$, let finite subsets $F_\D \subseteq \D$, 
$F_A\subseteq A$, and let a compact subset $K_G \subseteq G$
be given. Without loss of generality, we assume that
$F_A$ and $F_\D$ contain only contractions. 
Set $\widetilde{F_A}=F_A\cup \bigcup_{g\in K_G}\alpha_g(F_A)$, 
which is a compact subset of $A$.

By the equivalence between (i) and (ii) in Theorem~3.7
in~\cite{Sza_stronglyI_2018}, the $G$-algebras
$(A,\alpha)$ and $(A\otimes\D,\alpha\otimes\delta)$
are strongly cocycle conjugate. Using Lemma~4.2
in~\cite{Sza_stronglyI_2018},
find an isomorphism $\theta_A\colon A\otimes\D\to A$ satisfying
\begin{equation}\label{eqn:3.3}\tag{3.8}
\max_{g\in K_G}\|\alpha_g\circ\theta_A -\theta_A\circ(\alpha_g\otimes\delta_g)\|<\frac{\ep}{4}.
\end{equation}

Let $F_\D'\subseteq \D$ be a finite subset containing $F_\D$ and such that
for all $a\in \widetilde{F_A}$ there exist $n\in\N$ and elements $a_1,\ldots,a_n\in A$
and $d_1,\ldots,d_n\in F_\D'$ such that 
\begin{equation}\label{eqn:FDprime}\tag{3.9}
\Big\|a-\sum_{j=1}^n \theta_A(a_j\otimes d_j)\Big\|<\frac{\ep}{4}.
\end{equation}
Let $\psi\colon (\D\otimes \D,\delta\otimes\delta)\to (\D,\delta)$
be an equivariant isomorphism such that $\psi^{-1}$ is 
$G$-approximately 
unitarily equivalent to the first factor embedding, and 
let $v\in\U(\D)$ 
be a unitary satisfying $\max_{g\in K_G}\|\delta_g(v)-v\|<\ep/8$ 
and 
$\|v\psi(d\otimes 1)v^*-d\|<\ep/8$ for all $d\in F_\D'$.
Set $\theta_\D=\Ad(v)\circ\psi$. Given $g\in K_G$, we have
\begin{align*}\label{eqn:3.1}\tag{3.10}
\|\delta_g\circ\theta_\D-\theta_\D\circ(\delta\otimes\delta)_g\|&=
\|\delta_g\circ\Ad(v)\circ \psi -\Ad(v)\circ\psi\circ 
 (\delta\otimes\delta)_g\|\\
 &=\|\delta_g\circ\Ad(v)-\Ad(v)\circ\delta_g\|<\frac{\ep}{4}.
\end{align*}
Moreover, for $d\in F_\D$ we have
\begin{equation}\label{eqn:3.2}\tag{3.11}
\theta_\D(d\otimes 1)=v\psi(d\otimes 1)v^*\approx_{\ep/8}d.
\end{equation}

Define $\sigma\colon A\to A$ by setting
\[\sigma=\theta_A\circ (\id_A\otimes\theta_\D)\circ 
(\theta_A^{-1}\otimes 1).\]
Then $\sigma$ is a homomorphism.
Fix $a\in \widetilde{F_A}$. Find $n\in\N$ and elements $a_1,\ldots,a_n\in A$
and $d_1,\ldots,d_n\in F_\D'$ satisfying the inequality in (\ref{eqn:FDprime}).
Then
\begin{align*}\label{eqn:CloseId}\tag{3.12}
\sigma(a)&=\theta_A\circ (\id_A\otimes\theta_\D)(\theta_A^{-1}(a)\otimes 1))\\
&\stackrel{(\ref{eqn:FDprime})}{\approx_{\frac{\ep}{4}}} 
\sum_{j=1}^n\theta_A\circ (\id_A\otimes\theta_\D)(a_j\otimes d_j\otimes 1)\\
&\stackrel{(\ref{eqn:3.2})}{\approx_{\frac{\ep}{4}}} 
\sum_{j=1}^n\theta_A(a_j\otimes d_j)
\stackrel{(\ref{eqn:FDprime})}{\approx_{\frac{\ep}{4}}} 
a,
\end{align*}
thus establishing condition (5.b) in the statement. 
In order to check (5.d), let $g\in K_G$ and $a\in F_A$. 
Using that $\alpha_g(a)\in \widetilde{F_A}$, we have
\[\sigma(\alpha_g(a))\stackrel{(\ref{eqn:CloseId})}{\approx_{\frac{\ep}{4}}}
 \alpha_g(a)\stackrel{(\ref{eqn:CloseId})}{\approx_{\frac{\ep}{4}}}\alpha_g(\sigma(a)),
\]
as desired.
Define $\mu\colon \D\to A$ by 
\[\mu(d)=\theta_A(1_A\otimes \theta_\D(1\otimes d))\]
for all $d\in\D$. Then $\mu$ is a unital homomorphism. 
We proceed to check
condition (5.a) in the statement. 
Given $a\in A$ and $d\in \D$, we have 
\begin{align*}
\theta_A^{-1}(\sigma(a)\mu(d))&=
\big[(\id_A\otimes\theta_\D) 
(\theta_A^{-1}(a)\otimes 1)\big]
\big[(\id_A\otimes \theta_\D) (1_A\otimes 1_\D \otimes d)\big]\\
&= (\id_A\otimes\theta_\D)\big(\underbrace{(\theta_A^{-1}(a)\otimes 1) 
(1_A\otimes 1_\D\otimes d)}_{= 
(1_A\otimes 1_\D\otimes d)(\theta_A^{-1}(a)\otimes 1)}\big)
=\theta_A^{-1}(\mu(d)\sigma(a)),
\end{align*}
as desired.
Finally, to check condition (5.c), it suffices to take $a=1$.
Let $g\in K_G$ and $d\in F_\D$. Then
\begin{align*}
\mu(\delta_g(d))&=\theta_A(1\otimes\theta_\D(1\otimes\delta_g(d)))\\
&\stackrel{(\ref{eqn:3.1})}{\approx_{\frac{\ep}{4}}} 
\theta_A(1\otimes\delta_g(\theta_\D(1\otimes d)))\\
&\stackrel{(\ref{eqn:3.3})}{\approx_{\frac{\ep}{4}}}  
\alpha_g(\theta_A(1\otimes\theta_\D(1\otimes d))\\
&=\alpha_g(\mu(d)).
\end{align*}
We conclude that $\max_{g\in K_G}\|\mu(\delta_g(d))\alpha_g(a)-\alpha_g(\mu(d))\|<\ep\|a\|$ for all $a\in A$ and all $d\in F_\D$,
as desired. This finishes the proof.
\end{proof}

When working with norm-approximations, there is usually little to no difference
between considering finite subsets or norm-compact subsets of a given C*-algebra.
In particular, the sets $F_\D$ and $F_A$ in conditions (3), (4) and (5) of 
the previous theorem could equivalently be chosen to be compact. 
For the subset of $G$, the situation is a bit more subtle: one may 
replace the compact subset $K_G$ with a finite one, at the price of having
to control the continuity moduli on neighborhoods of the unit of $G$.
We included condition (4) in the above theorem because in practice
one is often only able to get approximations on a finite subset of $G$.
For example, most of the arguments one performs in the sequence
algebra will only be able to keep track of finitely many group elements at a
time. In such situations, it suffices to verify (almost) equivariance on a
finite subset as long as we have control over the relevant continuity moduli.

For future reference, we end this section with some comments about 
conditions (U) and (M) in the unital case.

\begin{rem}\label{rem:UnitalUCME}
If $A$ is unital, then the witness $a\in F_A$
in part~(3) of \autoref{thm:CharactDabs} 
is unnecessary in conditions (U) and (M) (see \autoref{nota:Abbreviations}); and they become 
\[ \mathrm{(U)} \ \  \|\psi(1)-1\|<\ep \ \ \ \ \ \ \ \ \ \ 
\ \ \ \ \ \ \ \mathrm{(M)}\ \ \|\psi(dd')-\psi(d)\psi(d')\|<\ep.\]
If, in addition, $A$ admits an $\alpha$-invariant
state, then condition (U) can be replaced by $\psi(1)=1$; see the comments after \autoref{thm:EqChoiEffros}.
\end{rem}



\section{Equivariant continuous bundles and absorption}
In this section, we study continuous bundles of $G$-algebras, 
also known as $G$-$C_0(X)$-algebras; see \autoref{df:GXalgebra}.
We are mostly interested in determining when such a $G$-$C_0(X)$-algebra
absorbs a given strongly self-absorbing $G$-algebra $(\D,\delta)$.
By Corollary~2.8 in~\cite{Sza_stronglyII_2018}, 
a necessary condition is that the fibers of the bundle 
absorb $(\D,\delta)$. When the base space $X$ has finite covering 
dimension and $(\D,\delta)$ satisfies a mild technical condition called
\emph{unitary regularity}, we show that the converse also holds. 
Hence, continuous bundles of $(\D,\delta)$-absorbing $G$-algebras
are themselves $(\D,\delta)$-absorbing. 
This is an equivariant version of Theorem~4.6 in~\cite{HirRorWin_algebras_2007}, 
and our proof depends crucially on the characterization
of $(\D,\delta)$-absorption obtained in Section~2.

The above mentioned result is known to fail if $X$ has infinite
covering dimension, already in the non-equivariant setting (that is,
for the trivial actions on $A$ and $\D$); see Example~4.8 
in~\cite{HirRorWin_algebras_2007}. 
However, even when $X$ is infinite-dimensional, if $(A,\alpha)$ is \emph{locally} $(\D,\delta)$-absorbing,
then $(A,\alpha)$ is $(\D,\delta)$-absorbing as well; see the 
comments at the end of this section.

We begin by introducing some necessary definitions. 
For a \ca\ $A$, we denote by $Z(A)$ its center.

\begin{df}\label{df:GXalgebra}
Let $A$ be a $C^*$-algebra and let $X$ be a locally 
compact Hausdorff space. We say that $A$ is a \emph{$C_0(X)$-algebra}, 
if there is a non-degenerate homomorphism $\rho \colon C_0(X)\to Z(M(A))$.
In order to lighten the notation, we usually omit $\rho$ and 
denote $\rho(f)(a)$ by $f\cdot a$ for $f \in C_0(X)$ and $a \in A$.

If $G$ is a locally compact group and $(A,\alpha)$ is a $G$-algebra,
then we say that it is a \emph{$G$-$C_0(X)$-algebra} if 
$C_0(X)\subseteq M(A)^G$.
\end{df}

In practice, we think of $G$-$C_0(X)$-algebras as 
being bundles over the base space $X$, where each fiber is naturally
a $G$-algebra. This is made precise as follows: 

\begin{nota}\label{nota:QuotientCXalg}
Let $G$ be a locally compact group, let $X$ be a locally compact
Hausdorff space, and let 
$(A,\alpha)$ be a $G$-$C_0(X)$-algebra. For a closed subset 
$Y \subseteq X$, we set 
$J_{Y}=  C_{0}(X\setminus Y) \cdot A$, which
is a $G$-invariant ideal in $A$. We let $A_Y$ denote the 
corresponding quotient, with quotient map $\pi_Y\colon A\to A_Y$.
Then $\alpha$ naturally induces an action $\alpha^Y$ on $A_Y$, 
and $\pi_Y$ is equivariant. Moreover, $(A_Y,\alpha^{Y})$ can be 
regarded either as a $G$-$C_0(X)$-algebra or as a 
$G$-$C_0(Y)$-algebra in a natural way.
\end{nota}

\begin{rem}\label{rem:FibersNorm}
When $Y$ is a singleton set $Y=\{x\}$, we abbreviate 
$A_Y$ to $A_{x}$ and $\alpha^Y$ to $\alpha^{(x)}$, and we call 
the $G$-algebra $(A_x,\alpha^{(x)})$ the \emph{fiber} of $(A,\alpha)$ over $x$.
For $a\in A$, we write $a_x$ for $\pi_x(a)$, and observe that
\[\|a\|=\sup_{x\in X}\|a_x\|.\]
Moreover, the function $x\mapsto \|a_x\|$ is upper semicontinuous, which
means that for every $\ep>0$, the set $\{x\in X\colon \|a_x\|<\ep\}$
is open in $X$. 
\end{rem}

We will prove \autoref{thm:FiberToBundleHRW} using
the new characterization of $(\D,\delta)$-absorption obtained 
in~\autoref{thm:CharactDabs}.
In particular, we will construct completely positive
contractive maps $\D\to A$ satisfying the conditions in part~(3) of 
said theorem. In doing so, we will first produce, for a closed subset
$Y\subseteq X$, maps $\D\to A$ which
satisfy said conditions \emph{when composed} with
the quotient map $\pi_Y\colon A\to A_Y$.
We isolate this notion in the following auxiliary definition
(see \autoref{nota:Abbreviations}).

\begin{df}\label{df:good}
Let $G$ be a second countable, locally compact group, 
let $(\D,\delta)$ be a strongly self-absorbing $G$-algebra, 
let $X$ be a locally compact Hausdorff space,
and let $(A, \alpha)$ be a separable $G$-$C_0(X)$-algebra. 
Let finite sets $F_\D\subseteq \D$ and $F_A \subseteq A$, and a compact subset $K_G\subseteq G$ be given. 
For a closed subset $Y\subseteq X$ and 
$\ep>0$, 
a completely positive contractive map $\psi\colon \D\to A$
is said to be \emph{$(F_\D, F_A, K_G, \varepsilon)$-regular 
for $Y$} if for every $y\in Y$, the composition
$\pi_y\circ \psi$ satisfies conditions (U), (C), (M) and 
(E) with respect to $(F_\D,\pi_y(F_A),K_G,\ep)$.
\end{df}

The above definiton is an adaptation to the equivariant setting of
Definition~4.2 in~\cite{HirRorWin_algebras_2007}, where 
condition (E) from \autoref{nota:Abbreviations} is incorporated
to bound the failure of equivariance for $\psi$ for $(F_\D,K_G)$.
We do not need to control the continuity modulus of $\alpha$ on
$\psi(F_\D)$, but for this it is necessary that we work with 
compact subsets of $G$, as opposed to just finite subsets.
Alternatively, if we worked instead with a finite subset $F_G\subseteq G$
and required condition (E) to hold for $(F_\D,F_G)$, we would
need to add a continuity-type condition over a compact 
neighborhood of the identity in $G$ along the lines of the 
displayed inequality in part~(4) of \autoref{thm:CharactDabs}.


\begin{rem}\label{rem:SufficesGoodAllX}
In the context of \autoref{df:good}, by \autoref{rem:FibersNorm}
a map $\psi\colon \D\to A$ is $(F_\D, F_A, K_G, \varepsilon)$-regular 
for all of $X$ if and only if it satisfies conditions (U), (C), 
(M), and (E) with respect to $(F_\D, F_A, K_G, \varepsilon)$.
In particular, it follows from
\autoref{thm:CharactDabs} that
$(A,\alpha)$ is $(\D,\delta)$-absorbing if and only if 
for all choices of $(F_\D, F_A, K_G, \varepsilon)$, 
there exists a $(F_\D, F_A, K_G, \varepsilon)$-regular map for
all of $X$.
\end{rem}

We begin by establishing the existence of regular maps in the 
presence of $(\D,\delta)$-absorbing fibers.

\begin{lma}\label{cor:key1}
Let the assumptions and notation be as in \autoref{df:good},
and suppose that there is $x\in X$ such that 
$(A_x,\alpha^{(x)})$ is $(\D,\delta)$-absorbing. 
Then there exist a compact neighborhood $Y$ of $x$
and a completely positive map $\psi\colon \D\to A$ which is 
$(F_\D, F_A,K_G,\ep)$-regular for $Y$.
\end{lma}
\begin{proof}
Since $(A_x,\alpha^{(x)})$ is $(\D,\delta)$-absorbing, it follows
from \autoref{thm:CharactDabs} that there exists a completely
positive contractive map 
$\varphi\colon \D \to A_x$ which satisfies conditions (U), (C), 
(M), and (E) with respect to the tuple
$(F_\D,\pi_x(F_A), K_G, \varepsilon)$. The map $\varphi$ is 
automatically nuclear, since so is $\D$.
Let $\psi\colon \D\to A$ be a lift of $\varphi$ as in the 
conclusion of \autoref{thm:EqChoiEffros} for the tuple
$(F_\D, K_G,\ep)$.
It is then immediate to check that $\psi$ is
$(F_\D, F_A,K_G,\ep)$-regular for $\{x\}$. By upper
semicontinuity of the norm function 
(see \autoref{nota:QuotientCXalg}), it follows that $\psi$
is regular not just for $\{x\}$, but for a neighborhood of it. 
The result then follows by local compactness of $X$.
\end{proof}

Our strategy to prove~\autoref{thm:FiberToBundleHRW} for 
$G$-$C(X)$-algebras is inspired by the arguments used in 
\cite{HirRorWin_algebras_2007}, and consists in 
``gluing''  
the regular maps obtained in~\autoref{cor:key1} to get a map 
which is regular for all of $X$. In general, there may be topological
obstructions to this gluing if the boundaries of the sets over which one 
wishes to glue have complicated homotopy groups. 
We will thus make some 
simplifications, which we proceed to explain.

A standard direct limit argument (see the beginning of the proof
of \autoref{thm:FiberToBundleHRW}) will allow us to assume that 
the base space $X$ is compact. In this case, 
if $X$ embeds homeomorphically 
into another compact space $Z$, then any $G$-$C(X)$-algebra 
$(A,\alpha)$ is naturally a 
$G$-$C(Z)$-algebra, by composing the original structure map with 
the quotient map $C(Z)\to C(X)$. With this structure, the fiber
of $(A,\alpha)$ over a point in $X$ (regarded as a point in $Z$) 
is the original one, and the 
fiber over a point in $Z\setminus X$ is the zero $G$-algebra. In
particular, all the $Z$-fibers are $(\D,\delta)$-stable, and 
are unital and separable if all the $X$-fibers are. 
Note, however, 
that $A$ is not in general a \emph{continuous} $C(Z)$-algebra, 
even if it is a continuous $C(X)$-algebra.

When $X$ is compact and finite-dimensional, then by Theorem~V.3 
in~\cite{HurWil_dimension_1941}, there 
exists $m\in\N$ such that $X$ embeds homeomorphically into the 
cube $[0,1]^m$. 
Using induction, we will see that the heart of the argument 
lies in establishing the case $m=1$. For $G$-$C([0,1])$-algebras,
the following implies that 
one only needs to worry about gluing regular maps over a single point. 

\begin{lma}\label{key1}
Let $G$ be a locally compact group, 
let $(\D,\delta)$ be a strongly self-absorbing $G$-algebra, 
let $(A, \alpha)$ be a $G$-$C([0,1])$-algebra such that 
$(A_t,\alpha^{(t)})$ is $(\D,\delta)$-absorbing for all $t\in [0,1]$, 
and let a tuple $(F_\D, F_A,K_G,\ep)$ as in \autoref{df:good}
be given. 
Then there exist $n\in\N$, points 
$0=t_{0} < t_{1} < \cdots <t_{n}=1$
and completely positive contractive maps 
$\psi_{1},\dots, \psi_{n} \colon \D\to A$ such that $\psi_{j}$ is 
$(F_\D, F_A,K_G,\ep)$-regular for 
$[t_{j-1}, t_{j}]$ for all $j=1,\ldots,n$.\end{lma}
\begin{proof} 
For every $t\in [0,1]$, use \autoref{cor:key1} to find a closed interval 
$Y_t$ continaing $t$ in its interior, and 
a map $\psi_t\colon \D\to A$ which is $(F_\D, F_A,K_G,\ep)$-regular 
for $Y_t$. Since the interiors of the $Y_t$ cover $[0,1]$, the result 
follows by compactness.
\end{proof}

Gluing the regular maps obtained above will require better 
control over the norm estimates in (U), (C), (M), and (E) from 
\autoref{df:good} over the endpoints;
see 
the proof of \autoref{key3}.
The following notion will be helpful in this task.

\begin{df}\label{df:better}
Let $G$ be a locally compact group, 
let $(\D,\delta)$ be a strongly self-absorbing $G$-algebra, 
let $(A, \alpha)$ be a unital, separable $G$-$C([0,1])$-algebra, and let 
$Y$ be a closed subinterval in $[0,1]$. 
Let finite subsets $F_\D\subseteq F_\D'\subseteq \D$, 
$F_A\subseteq A$, a compact subset
$K_G\subseteq G$, 
and tolerances $0<\varepsilon'< \varepsilon$ be given. 
We say that a completely positive
contractive map $\psi\colon \D\to A$ is 
\emph{$(F_\D, F_A, K_G, \varepsilon, F_\D', \varepsilon')$-regular for $Y$}, if $\psi$ is 
$(F_\D, F_A, K_G,\varepsilon)$-regular for $Y$, and there is a 
compact neighborhood of the endpoints of $Y$
for which $\psi$ is $(F_\D',F_A, K_G,\varepsilon')$-regular.
\end{df}

In the context of the above definition, and since $A$ is assumed
to be unital, following \autoref{rem:UnitalUCME} we will use
conditions (U) and (M) without the multiplicative witness $a\in F_A$. 

The gluing procedure explained after \autoref{cor:key1} will be
accomplished in the following two lemmas. In the first one, 
we show that, given two regular maps $\psi_1$ and $\psi_2$ that 
we wish to glue at $s\in [0,1]$, we can find two other regular maps $\nu_1$ and 
$\nu_2$, such that, one localized at $s$, the maps $\psi_j$ and $\nu_j$ approximately commute
for $j=1,2$, and $\nu_1$ is close to $\nu_2$. 

We follow the arguments in~\cite{HirRorWin_algebras_2007}, 
but stress the fact that working in the equivariant setting makes 
the arguments necessarily more technical, since we must keep 
track of equivariance
and (at least implicitly) continuity. 
We omit those computations similar to 
the ones in \cite{HirRorWin_algebras_2007}, and give details for 
everything related to the actions.


\begin{lma}\label{key2}
Let the notation and assumptions be as in \autoref{df:better}.
Let $r,s,t\in [0,1]$ with $r<s<t$, and let 
$\psi_1,\psi_2 \colon \D \to A$ 
be completely positive contractive maps which are 
$(F_\D, F_A, K_G,\ep)$-regular for $[r,s]$ and $[s,t]$, 
respectively. If $(A_{s}, \alpha^{(s)})$ is 
$(\D , \delta)$-stable, 
then there exist completely positive contractive maps 
\[\varphi_1,\varphi_2,\nu_1, \nu_2 \colon \D \to A \ \mbox{ and } \ \mu_1, \mu_2 \colon \D  \otimes \D \to A\] 
such that the following are satisfied for $j=1,2$ and $d,d'\in F_\D$:
\be
\item $\varphi_1$ and $\varphi_2$ are $(F_\D, F_A, K_G,\ep)$-regular for $[r,s]$ and $[s,t]$, respectively;
\item $\nu_j$ is $(F_\D, F_A, K_G,3\ep)$-regular for $\{s\}$;
\item $\|\varphi_j(d)_s\nu_j(d')_s-\nu_j(d')_s\varphi_j(d)_s\| < 2\ep$;
\item $\|\varphi_j(d)_s \nu_j(d')_s-\mu_j(d \otimes d'))_s\| <\ep$;
\item $\|\alpha_g^{(s)}(\mu_j(d\otimes d')_s)-
\mu_j(\delta_g(d)\otimes\delta_g(d'))_s\|<\ep$ for all $g\in K_G$;
\item $\|\nu_1(d)_s - \nu_2(d)_s\| < 2 \ep$.
\ee

If $\psi_1$ and $\psi_2$ are 
$(F_\D, F_A, K_G,\ep,F_\D',\ep')$-regular for $[r,s]$ and $[s,t]$, respectively, 
then the same can be arranged for $\varphi_1$ and $\varphi_2$. Moreover,
one can arrange that $\nu_1$ and $\nu_2$ are $(F_\D',F_A, K_G, 3\ep')$-regular for $\{s\}$, and that
conditions (3) through (6) above hold if one replaces 
$F_\D$ with $F_\D'$ and $\ep$ with $\ep'$. 
\end{lma}
\begin{proof}
Without loss of generality, we assume that $F_\D$ and $F_A$ consist
of positive contractions, that $F_\D$ contains the unit of $\D$, that $F_A$ contains the unit of $A$,
and that $K_G$ contains the unit of $G$. 
Using that $F_A$ and $F_\D$ are finite, and that $K_G$ is compact, find 
$0< \ep_0 < \ep$ such that 
$\psi_1$ and $\psi_2$ are $(F_\D, F_A, K_G,\ep_0)$-regular 
for $[r,s]$ and $[s,t]$, respectively. Explicitly, with
$I_1=[r,s]$ and $I_2=[s,t]$, for $j=1,2$ we have (see \autoref{rem:UnitalUCME}):
\be
\item[(U)$_j$] $\|\psi_j(1)_x-1_x\|<\ep_0$ for all $x\in I_j$
\item[(C)$_j$] $\|a_x\psi_j(d)_x-\psi_j(d)_xa_x\|<\ep_0$ for all $a\in F_A$, $d \in F_\D$ and $x\in I_j$;
\item[(M)$_j$] $\|\psi_j(dd')_x-\psi_j(d)_x\psi_j(d')_x\|<\ep_0$
for all $d,d'\in F_\D$ and $x\in I_j$;
\item[(E)$_j$] $\max\limits_{g\in K_G}\|\alpha^{(x)}_g(\psi_j(d)_x)-\psi_j(\delta_g(d))_x\|<\ep_0$
for $d\in F_\D$ and $x\in I_j$.
\ee
Equivalently, $\pi_x\circ\psi_j$ satisfies (U), (C), (M)
and (E) with respect to $(F_\D, F_A, K_G,\ep_0)$. 
Set $\widetilde{\ep}= (\ep-\ep_0)/6$, set $\widetilde{F_\D}=\bigcup_{g\in K_G}\delta_g(F_\D)$, which is a compact subset
of $\D$, and set
\[\widetilde{F_A}=F_A\cup \psi_1(\widetilde{F_\D})\cup \psi_1(\widetilde{F_\D})\psi_1(\widetilde{F_\D})\cup \psi_2(\widetilde{F_\D})\cup \psi_2(\widetilde{F_\D})\psi_2(\widetilde{F_\D}),\]
which is a compact subset of $A$. 
Since 
$(A_{s}, \alpha^{(s)})$ is $(\D, \delta)$-stable and unital, 
it satisfies the conditions in part~(5) of \autoref{thm:CharactDabs}.
We thus fix homomorphisms
$\mu\colon \D \to A_s$ and $\sigma\colon A_s\to A_s$
with $\mu(1)=1$, satisfying:
\be
\item[(5.a)] $\sigma(a_s)\mu(d)=\mu(d)\sigma(a_s)$ for all $d\in\D$
and all $a\in A$;
\item[(5.b)] $\|\sigma(a_s)-a_s\|<\widetilde{\ep}$ for all $a\in \widetilde{F_A}$;
\item[(5.c)] $\max\limits_{g\in K_G}\|\alpha_g^{(s)}(\mu(d))-\mu(\delta_g(d))\|<\widetilde{\ep}$ for $d\in F_\D$; and 
\item[(5.d)] $\max\limits_{g\in K_G}\|\sigma(\alpha^{(s)}_g(a_s))-\alpha^{(s)}_g(\sigma(a_s))\|<\widetilde{\ep}$ for all $a\in \widetilde{F_A}$.
\ee

Fix $j=1,2$.
Using (twice) that $\sigma\circ\pi_s\approx_{\widetilde{\ep}}\pi_s$
on $F_A\cup \psi_j(F_\D)$ by condition (5.b) above, one easily shows that 
$\sigma\circ\pi_s\circ\psi_j$
satisfies conditions analogous to (U)$_j$, (C)$_j$, (M)$_j$, and
(E)$_j$, with $\ep_0+2\widetilde{\ep}$ 
instead of $\ep_0$. We check this for (E)$_j$, and leave the other ones
to the reader. Let $d\in F_\D$ and let $g\in K_G$. Using that
$\delta_g(d)\in \widetilde{F_\D}$ at the fourth step, we get
\begin{align*}
\alpha^{(s)}_g((\sigma\circ\pi_s\circ\psi_j)(d))
&=\alpha^{(s)}_g(\sigma(\psi_j(d)_s))
\stackrel{(5.\mathrm{b})}{\approx_{\widetilde{\ep}}}
\alpha^{(s)}_g(\psi_j(d)_s)
\stackrel{\mathrm{(E)}_j}{\approx_{\ep_0}}\psi_j(\delta_g(d))_s\\
&\stackrel{(5.\mathrm{b})}{\approx_{\widetilde{\ep}}}
\sigma(\psi_j(\delta_g(d))_s)
=(\sigma\circ\pi_s\circ\psi_j)(\delta_g(d)),
\end{align*}
as desired.
By upper semicontinuity of the norm,
there exists $l>0$ with $[s-l,s+l]\subseteq I_1\cup I_2$ 
such that $\sigma\circ\pi_x\circ\psi_j$ satisfies
similar estimates 
for all $x\in [s-l,s+l]$. 
Let $\widetilde{\psi}_j\colon \D\to A$
be a completely positive contractive lift of 
$\sigma\circ\pi_s\circ\psi_j$. Then $\widetilde{\psi}_j$
is $(F_\D, F_A,K_G, \ep_0+2\widetilde{\ep})$-regular for $[s-l,s+l]$.
Moreover, for all $d\in F_\D$ we have 
\begin{equation}\label{widtildesigma}\tag{4.1}
\widetilde{\psi}_j(d)_s=\sigma(\psi_j(d)_s)\stackrel{\mathrm{(5.b)}}{\approx_{\widetilde{\ep}}} \psi_j(d)_s,
\end{equation}
so by reducing $l$ we may also assume that $\|\widetilde{\psi}_j(d)_x-\psi_j(d)_x\|<\widetilde{\ep}$ for all $d\in F_\D$ and all $x\in [s-l,s+l]$.
Let $f_1\colon \mathbb{R}\to [0,1]$ be given by the following graph:
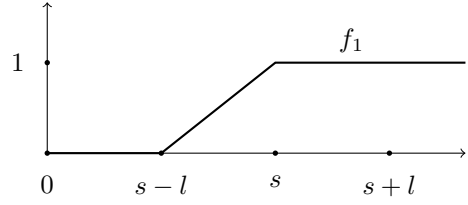
\begin{center}\begin{tikzpicture}
\draw[->] (0,0)--(5.5,0) node[anchor=east]{};
{
	\fill (0,0) circle (1pt) node [below=5pt] {0};
}
{
	\fill (1.5,0) circle (1pt) node [below=5pt] {$s-l$};
}
{
	\fill (3,0) circle (1pt) node [below=5pt] {$s$};
}
{
	\fill (4.5,0) circle (1pt) node [below=5pt] {$s+l$};

}
\draw[->] (0,0)--(0,2) node[anchor=north]{};
{
	\fill (0,1.2) circle (1pt) node [left=5pt] {1};
}
\draw[thick] [ black] (0,0)--(1.5, 0)--(3,1.2)--(4.5,1.2)--(5.5,1.2);
\draw [black](4,1.5) node{$f_{1}$};
\end{tikzpicture}
\end{center}
and set $f_2(s+x)=f_1(s-x)$ for all $x\in \mathbb{R}$.
We regard $f_1$ and $f_2$ as elements in $C([0,1])$ 
by restricting them to $[0,1]$.
For $j=1,2$, define a completely positive contractive map 
$\varphi_j \colon \D \to A$ by
\begin{equation}\label{rho'}\tag{4.2}
\varphi_j(d)=(1-f_j) \cdot  \psi_j(d)+f_j \cdot \widetilde{\psi_j}(d)\end{equation}
for all $d\in \D$.

\textbf{Claim 1:} \emph{$\varphi_j$ is $(F_\D, F_A, K_G,\ep_0+2\widetilde{\ep})$-regular for $I_j$}. 
We prove this for $j=1$, since the proof for $j=2$ is analogous.
Note that $\varphi_1$ agrees with $\psi_1$ on $[0,s-l]$, and 
with $\widetilde{\psi}_j$ on $[s,s+l]$.  Thus, it suffices
to prove regularity of $\varphi_1$ for $[s-l,s]$.
Conditions (U), (C) and (M) are explicitly verified in Lemma~4.4 
of~\cite{HirRorWin_algebras_2007}, so we only prove (E).
Given $d\in F_\D$, $g\in K_G$ and $x\in [s-l,s]$, we have
\begin{align*}\alpha_g^{(x)}(\varphi_1(d)_x)
 &= 
 (1-f_1(x))\underbrace{\alpha_g^{(x)}(\psi_1(d)_x)}_{\approx_{\ep_0} \
 \psi_1(\delta_g(d))_x}+f_1(x)\underbrace{\alpha_g^{(x)}(\widetilde{\psi}_1(d)_x)}_{\approx_{\ep_0+2\widetilde{\ep}} \ 
 \widetilde{\psi}_1(\delta_g(d))_x} \\
 &\approx_{\ep_0+2\widetilde{\ep}} \left((1-f_1(x))\psi_1(\delta_g(d))_x+f_1(x)\widetilde{\psi}_1(\delta_g(d))_x\right)\\
 &=\varphi_1(\delta_g(d))_x,
\end{align*}
where the bracketed approximations use condition (E)$_1$ and the 
fact that $\widetilde{\psi}_1$ is $(F_\D, F_A, K_G,
\varepsilon_{0}+2\widetilde{\varepsilon})$-regular 
for $[s-l,s+l]$, respectively. 
We deduce that $\varphi_1$ satisfies condition
(E), as desired. 
This proves the claim, and establishes condition (1) in the statement.

We proceed to the construction of the map $\mu_j$. 
Since 
\[\pi_s\circ\varphi_j=f_j(s)\pi_s\circ\widetilde{\psi}_j=\sigma\circ\pi_s\circ\psi_j,\]
it follows that 
the range of $\pi_s\circ\varphi_j$ is contained
in the range of $\sigma$, and thus by (5.a) we have 
\[\label{eqn:varphiscommute}\tag{4.3}
\varphi_j(d)_s\mu(d')=\mu(d')\varphi_j(d)_s
\]
for all $d,d'\in\D$. Thus,
there exists a completely positive contractive map 
$\overline{\mu}_j\colon \D  \otimes \D \to A_{s}$ which is
determined by
$\overline{\mu}_j(d\otimes d')= \varphi_j(d)_{s} \mu(d')$
for all $d,d'\in\D$. Let $\overline{\nu}_j\colon \D\to A_s$ be given
by $\overline{\nu}_j(d)=\overline{\mu}_j(1\otimes d)$ for all $d\in\D$.
Use the Choi-Effros lifting theorem to find 
completely positive contractive maps 
\[\mu_j\colon \D \otimes \D \to A \ \ \mbox{ and } \ \ \nu_j\colon \D\to A\]
satisfying $\pi_s\circ\mu_j=\overline{\mu}_j$ and 
$\pi_s\circ\nu_j=\overline{\nu}_j$. 

\textbf{Claim 2:} \emph{$\pi_s\circ \nu_j$ satisfies conditions (U), (C), (M), and (E) with respect to $(F_\D, \pi_s(F_A), K_G, 3\ep)$.}
As in Claim~1, conditions (U), (C) and (M) 
are checked in the proof of Lemma~4.4 in~\cite{HirRorWin_algebras_2007},
so we only prove (E).
By Claim~1, condition (U) for $\varphi_j$ gives 
$\varphi_j(1)_s\approx_{\ep_0+2\widetilde{\ep}}1$, Using this 
at the last step,
we get
\[\label{eqn:nujs}\tag{4.4}
\nu_j(d)_s=\overline{\nu}_j(d)=\overline{\mu}_j(1\otimes d)=\varphi_j(1)_{s} \mu(d)\approx_{\ep} \mu(d)\]
for all $d\in \D$. For $g\in K_G$ and $d\in F_\D$, we have
\begin{align*}
\alpha_g^{(s)}(\nu_j(d)_s)&\stackrel{(\ref{eqn:nujs})}
{\approx}_{\hspace{-.15cm}\ep}\alpha_g^{(s)}(\mu(d))\\
&\stackrel{\mathrm{(5.c)}}{\approx_{\widetilde{\ep}}} 
\mu(\delta_g(d))\\
&\stackrel{(\ref{eqn:nujs})}
{\approx}_{\hspace{-.15cm}\ep} \nu_j(\delta_g(d))_s.
\end{align*}
This proves the claim, and establishes condition (2) in the statement.


We check condition (3). Let $d,d'\in F_\D$, and recall that $F_\D$
was assumed to consist of contractions. Then
\begin{align*}
\varphi_j(d)_s\nu_j(d')_s&\stackrel{(\ref{eqn:nujs})}
{\approx}_{\hspace{-.15cm}\ep}
\varphi_j(d)_s\mu(d')\\
&\stackrel{(\ref{eqn:varphiscommute})}{=} \mu(d')\varphi_j(d)_s\\
&\stackrel{(\ref{eqn:nujs})}
{\approx}_{\hspace{-.15cm}\ep}\nu_j(d')_s\varphi_j(d)_s,
\end{align*}
as desired. To check Condition (4), let $d,d'\in F_\D$. Then
\begin{align*}
\mu_j(d\otimes d')_s&=\varphi_j(d)_s\mu(d') 
\stackrel{(\ref{eqn:nujs})}
{\approx}_{\hspace{-.15cm}\ep}=\varphi_j(d)_s\nu_j(d')_s,
\end{align*}
as desired.
We turn to condition (5), so fix $d,d'\in F_\D$, 
and $g\in K_G$. Then
\begin{align*}
\alpha_g^{(s)}(\mu_j(d\otimes d')_s)&=
\alpha_g^{(s)}(\mu(d')\varphi_j(d)_s)\\
&\stackrel{\mathrm{(5.c)}}{\approx_{\widetilde{\ep}}} 
\mu(\delta_g(d'))\alpha_{g}^{(s)}(\varphi_j(d)_s)\\
&\stackrel{\mathrm{Claim~\!1}}{\approx_{\ep_0+2\widetilde{\ep}}}
\mu(\delta_g(d'))\varphi_j(\delta_g(d))_s\\
&=\mu_j(\delta_g(d)\otimes\delta_g(d'))_s,
\end{align*}
as desired. To prove condition (6), let $d\in F_\D$. Then
\begin{align*}
\nu_1(d)_s&\stackrel{(\ref{eqn:nujs})}
{\approx}_{\hspace{-.15cm}\ep} \mu(d)\stackrel{(\ref{eqn:nujs})}
{\approx}_{\hspace{-.15cm}\ep} \nu_2(d)_s
\end{align*} 
as desired. 
This finishes the proof of the first part of the lemma.
  
The last assertion in the lemma is proved identically as above, noticing
that all the estimates we made at the point $\{s\}$ only depend on
the regularity of $\psi_1$ and $\psi_2$ at $\{s\}$.
We omit the details.\end{proof}
  
We will need the following notation, which is borrowed
from \cite{Sza_stronglyI_2018}.
  
\begin{nota}
Let $G$ be a second countable, locally compact group and 
let $(D,\delta)$ be a unital $G$-algebra. 
Following \cite{Sza_stronglyI_2018}, 
for a compact set $K \subseteq G$ and $\ep >0$, we set 
\[ 
D_{\ep, K}^{\delta} = \{d \in D \colon \|\delta_g(d)-d\| < \ep \mbox{ for  all } g \in K\}.
\]
and $\U(D_{\ep, K}^{\delta})= \U(D) \cap D_{\ep, K}^{\delta}$. 
Finally, we let $\U_{0}(D^{\delta}_{K, \ep})$ denote the 
connected component of $1_D$ in $\U(D_{\ep,K}^\delta)$.
\end{nota}

We now recall a technical condition for unital $G$-algebras called \emph{unitary 
regularity}; see Definition~2.17 in~\cite{Sza_stronglyII_2018}. For the trivial
$G$-action on $D$, the property reduces to the fact that the 
commutator subgroup of $\mathcal{U}(D)$ is contained in $\mathcal{U}_0(D)$. 

\begin{df}\label{df:UnitReg}
Let $G$ be a locally compact group, and 
let $(D,\delta)$ be a unital $G$-algebra. We say that $\delta$ is \emph{unitarily regular}, if for every compact set $K \subseteq G$ and $\varepsilon >0$, there exists $\gamma >0$ such that the commutator
$uvu^{*}v^{*}$ belongs to $\mathcal{U}_{0}(D_{\varepsilon, K} ^{\delta})$ for every $ u,v\in \mathcal{U}(D_{\gamma, K} ^{\delta})$.
\end{df}

If $G$ is compact, then the above definition reduces to the assertion that 
the commutator subgroup of $\mathcal{U}(D^\delta)$ is contained in 
$\mathcal{U}_0(D^\delta)$. 
Recall that a \ca\ $D$ is said to be \emph{$K_1$-injective} if the 
canonical map $\U(D)/\U_0(D)\to K_1(D)$ is injective. By Proposition~2.19
in~\cite{Sza_stronglyII_2018}, a unital $G$-algebra $(D,\delta)$ is
unitarily regular whenever the fixed point algebra of the sequence algebra 
$D_{\I}$ is
$K_1$-injective. Using this, it is shown in said 
proposition that any action absorbing
$\id_{\mathcal{Z}}$ tensorially is unitarily regular. 

In the equivariang setting, unitary regularity serves as a replacement of
$K_1$-injectivity for strongly self-absorbing $G$-actions. Even though 
$K_1$-injectivity for strongly self-absorbing \emph{$C^*$-algebras}
is automatic by the main result of \cite{Win_strongly_2011}, the analogous
result for strongly self-absorbing $G$-algebras remains open (although it 
has been confirmed when $G$ is amenable; see \cite{GarHir_strongly_2018}).

For a subset $F$ of a \ca\ and $n\in\N$, 
we let $F^n$ denote the set of all $n$-fold products of elements
in $F$.

\begin{lma}\label{key3}
Let $G$ be a second countable, locally compact group, 
let $(\D,\delta)$ be a strongly self-absorbing, unitarily regular 
$G$-algebra, let $F_\D\subseteq \D$ be a finite subset, let
$K_G\subseteq G$ be a compact subset,
and let $\ep>0$ be given.
Then there exist a finite set 
$F_\D'\supseteq F_\D$ and $0<\ep'<\ep$ with the following property:

Let $(A, \alpha)$ be a separable, 
unital $G$-$C([0,1])$-algebra with $(\D,\delta)$-stable fibers, let $r<s<t$ in $[0,1]$, let 
$F_A\subseteq A$ be a finite subset,
and let $\psi_1,\psi_2\colon \D\to A$ be completely positive contractive maps 
that are $(F_\D,F_A,K_G, \ep,F_\D',\ep')$-regular
on $[r,s]$ and $[s,t]$, respectively. Then
there is a completely positive contractive map
$\Psi\colon \D\to A$ which is 
$(F_\D,F_A,K_G,\ep,F_\D',\ep')$-regular for $[r,t]$.
\end{lma}
\begin{proof} Without loss of generality, we assume that 
$F_\D$ contains $1_\D$ and that $K_G$ contains $1_G$.
Set $\widetilde{F_\D}= \bigcup_{g\in K_G}\delta_g(F_\D\cdot F_\D)$, 
which is a compact subset of $\D$.
Using that $(\D,\delta)$ is unitarily regular and self-absorbing,
we apply Lemma~3.10 of~\cite{Sza_stronglyII_2018} to find 
a continuous unitary path  
$u \colon [0,1]\to \U(\D\otimes \D)$ with $u_{0}=1_{\D \otimes \D}$ satisfying
$\max_{g\in K_G}\|(\delta\otimes\delta)_g(u_x)-u_x\|<\ep/9$ for all $x\in [0,1]$ 
and such that
\begin{equation}\label{eqn:ufliposta}\tag{4.5}
\|u_{1}(d \otimes 1)u_{1}^{*}- 1 \otimes d\| < \frac{\ep}{9}
\end{equation}
for all $x\in [0,1]$ and all $d \in \widetilde{F_\D}$. 
Using that $(u_{x})_{x \in [0,1]}$ is norm-compact, find 
$m\in \mathbb{N}$ and $v_{k,\ell}, w_{k,\ell}\in \D$, for 
$k,\ell=1,\dots,m$, such that, with $y_{k}= \sum _{\ell=1} ^{m}  v_{k,\ell}\otimes w_{k,\ell}$, for any $x \in [0,1]$ there is $k\in\{1,\ldots,m\}$ with
$\|u_x-y_k\|<\ep/9$. 
Without loss of generality, we may assume that $\|y_k\|\leq 1$.
Note that 
\begin{equation}\label{eqn:u1flip}\tag{4.6}
\max_{g\in K_G}\|(\delta\otimes\delta)_g(y_k)-y_k\|<\frac{2\ep}{9}
\end{equation}
for all $k=1,\ldots,m$. Set $\ep'= \ep/144m^{4}$.
With $V=\{v_{k,\ell}, w_{k,\ell}\colon k,\ell =1,\dots m\}$, set 
\[
F_\D'=\bigcup_{g\in K_G}\delta_g((F_\D\cup V\cup V^*)^6).\]
Then $F_\D\subseteq \widetilde{F_\D}\subseteq F_\D'$.

Let $(A, \alpha)$ be a separable, 
unital $G$-$C([0,1])$-algebra with $(\D,\delta)$-stable fibers, let $r<s<t$ in $[0,1]$, let 
$F_A\subseteq A$ be a finite subset,
and let $\psi_1,\psi_2\colon \D\to A$ be completely positive contractive maps 
that are $(F_\D,F_A,K_G, \ep,F_\D',\ep')$-regular
on $[r,s]$ and $[s,t]$, respectively. 
Without loss of generality, we assume that $F_A$ contains
the unit of $A$.
For $j=1,2$, let $\varphi_j,\nu_j\colon \D\to A$ and $\mu_j\colon
\D\otimes\D\to A$ be completely positive contractive maps
as in the conclusion of \autoref{key2}.
Using upper semicontinuity of the norm,
and since $\varphi_1$ and $\varphi_2$ 
are $(F_\D',F_A, K_G,\ep')$-regular for $\{s\}$,
find $l>0$ such that, with $I=[s-3l,s+3l]$, 
we have $I\subseteq [s,t]$, the maps
$\varphi_1, \varphi_2 $ are $(F_\D',F_A, K_G,\ep')$-regular for $I$
and 
conditions (3) through (6) in \autoref{key2}
hold not just for $s$, but for all $x\in I$.
In particular, condition (5) 
in \autoref{key2} gives
\begin{align*}\label{eqn:mujEquiv}\tag{4.7}
\max_{g\in K_G}\|\alpha^{(x)}_g(\mu_j(d\otimes d')_x)-
\mu_j((\delta\otimes\delta)_g(d\otimes d'))_x\|<\ep'
\end{align*}
for all $x\in I$ and all $d,d'\in F_\D'$.
For $j=1,2$, define a completely positive contractive map 
$\sigma_j \colon C([0,1] )\otimes \D \otimes \D\to A$ by
\[
\sigma_j(f \otimes d \otimes d')=f \cdot \mu_j(d \otimes d')
\]
for all $f\in C([0,1])$ and $d,d'\in\D$. 
It is then immediate to check that
\begin{align}\label{eqn:sigmaCXmap}\tag{4.8}
\sigma_j(\eta)_x=\mu_j(\eta(x))_x 
\end{align}
for all $\eta\in C([0,1])\otimes \D\otimes\D$ and all $x\in [0,1]$.
Let $f_{1}, h_{1}\colon \mathbb{R}\to [0,1]$ be given by the following
graphs:
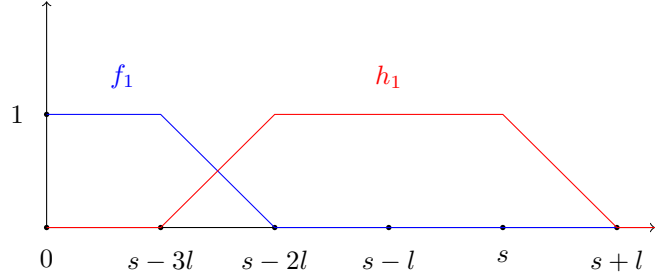
\begin{center}
\begin{tikzpicture}
\draw[->] (0,0)--(8,0) node[anchor=east]{};
{
	\fill (0,0) circle (1pt) node [below=5pt] {0};
}
{
	\fill (1.5,0) circle (1pt) node [below=5pt] {$s-3l$};
}
{
	\fill (3,0) circle (1pt) node [below=5pt] {$s-2l$};
}
{
	\fill (4.5,0) circle (1pt) node [below=5pt] {$s-l$};
}
{
	\fill (6,0) circle (1pt) node [below=5pt] {$s$};
}
{
	\fill (7.5,0) circle (1pt) node [below=5pt] {$s+l$};
}
\draw[->] (0,0)--(0,3) node[anchor=north]{};
{
	\fill (0,1.5) circle (1pt) node [left=5pt] {1};
}
\draw [blue](0,1.5)--(1.5,1.5)--(3,0)--(4.5,0)--(6,0)--(7.5,0);
\draw[blue] (1, 2) node{$f_{1}$};
\draw[red] (0,0)--(1.5,0)--(3,1.5)--(4.5,1.5)--(6,1.5)--(7.5,0)--(8,0);
\draw [red](4.5,2) node{$h_1$};
\end{tikzpicture}
\end{center}

Set $f_2(s+x)=f_1(s-x)$ and $h_2(s+x)=h_1(s-x)$ for all 
$x\in \mathbb{R}$.
Note that $f_1+h_2+h_2+f_2=1$ and that 
$f_1(f_2+h_2)=0=f_2(f_1+h_1)$. 
Let $\xi_1 \colon \mathbb{R}\to [0,1]$ be given by
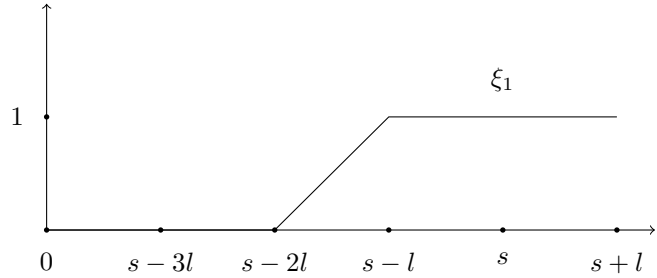
\begin{center}
\begin{tikzpicture}
\draw[->] (0,0)--(8,0) node[anchor=east]{};
{
	\fill (0,0) circle (1pt) node [below=5pt] {0};
}
{
	\fill (1.5,0) circle (1pt) node [below=5pt] {$s-3l$};
}
{
	\fill (3,0) circle (1pt) node [below=5pt] {$s-2l$};
}
{
	\fill (4.5,0) circle (1pt) node [below=5pt] {$s-l$};

}
{
	\fill (6,0) circle (1pt) node [below=5pt] {$s$};
}
{
	\fill (7.5,0) circle (1pt) node [below=5pt] {$s+l$};
}
\draw[->] (0,0)--(0,3) node[anchor=north]{};
{
	\fill (0,1.5) circle (1pt) node [left=5pt] {1};
}
\draw [black](0,0)--(1.5,0)--(3,0)--(4.5,1.5)--(6,1.5)--(7.5,1.5);
\draw[black] (6, 2) node{$\xi_1$};
\end{tikzpicture}
\end{center}
and set $\xi_2(s+x)=\xi_1(s-x)$ for all $x\in \mathbb{R}$.
We regard $f_1, f_2, h_1,h_2, \xi_1$ and $\xi_2$
as elements in $C([0,1])$ by restricting them to $[0,1]$.
For $j=1,2$, set 
\[v_j=u\circ \xi_j \colon [0,1]\to \U(\D \otimes \D),\] 
and let $\theta_j \colon \D\to A$ be the completely positive contractive map defined by
\[
\theta_j(d)=(\sigma_j \circ \Ad(v_j))(1_{C([0,1])} \otimes d \otimes 1)
\]
for all $d\in\D$. 
Finally, define $\Psi \colon \D\to A$ by
\[
\Psi(d)= f_1\varphi_1(d)+h_1\theta_1(d)+h_2\theta_2(d)+f_2\varphi_2(d)
\]
for all $d\in\D$. We claim that $\Psi$ is $(F_\D, F_A, K_G,\ep,F_\D',\ep')$-regular for $[r,t]$. 

Observe that $\Psi(d)_{[0,s-3l]}=\varphi_1(d)_{[0,s-3l]}$ and $\Psi(d)_{[s+3l,1]}=\varphi_2(d)_{[s+3l,1]}$. 
Since $\varphi_1$ and $\varphi_2$ are $(F_\D, F_A, K_G,\ep,F_\D',\ep')$-regular for $[r,s]$ and $[s,t]$, respectively, it suffices to prove that $\Psi$ is 
$(F_\D, F_A,K_G,\ep)$-regular for 
$[s-3l, s+3l]$.

The proof of Lemma~4.5 in~\cite{HirRorWin_algebras_2007}
shows that $\Psi$ satisfies 
conditions (U), (C), and (M) with respect to
$(F_\D,F_A,\ep)$ for $[s-3l,s+3l]$. Thus, it suffices to check
the following condition for all $d\in F_\D$ and 
for all $x\in [s-3l,s+3l]$:
\[\tag*{(E)$_x$}\max\limits_{g\in K_G}\|\alpha_g^{(x)}(\Psi(d)_x)-\Psi(\delta_g(d))_x\|<\ep.\]

We divide the proof into the following cases:

\underline{Case I:} $x \in [s-3l,s-2l]$. Then $v_1(x)=1$, and thus 
$\theta_1(d)_{x}=\varphi_1(d)_{x}$ for all $d\in\D$. 
It follows that 
$\Psi(d)_x=f_1(x)\varphi_1(d)_x+h_1(x)\varphi_1(d)_x=\varphi_1(d)_x$,
as $f_1+h_1=1$ on $[s-3l,s-2l]$. Since $\varphi_1$ satisfies
condition (E)$_x$ 
above, the result follows in this case.

\underline{Case II:}
$x \in [s-2l, s-l]$. Then $f_1(x)=f_2(x)=h_2(x)=0$ and $h_1(x)=1$.
Hence $\Psi(d)_x=\theta_1(d)_x$ for all $d\in\D$.
Find $k\in\{1,\ldots,m\}$ with $\|v_1(x)-y_k\|<\ep/9$, and 
note that $y_k(d\otimes 1)y_k^*$ belongs to $F_\D'\otimes F_\D'$
for all $d\in F_\D$ (recall that $F_\D$ contains $1_\D$).
For $d\in \D$, it follows that
\begin{align*}\label{eqn:CaseII}\tag{4.9}
\Psi(d)_x&=
\sigma_1(v_1(1\otimes d\otimes 1)v_1^*)_x\\
&\stackrel{(\ref{eqn:sigmaCXmap})}{=}
\mu_1(v_1(x)(d\otimes 1)v_1(x)^*)_x\\
&\approx_{\frac{2\ep}{9}}
\mu_1(y_k(d\otimes 1)y_k^*)_x
\end{align*}

For $g\in K_G$ and $d\in F_\D$,
we get
\begin{align*}
\alpha_g^{(x)}(\Psi(d)_x)
&\stackrel{(\ref{eqn:CaseII})}{\approx_{\frac{2\ep}{9}}}
\alpha_g^{(x)}(\mu_1(y_k(d\otimes 1)y_k^*)_x)\\
&\stackrel{(\ref{eqn:mujEquiv})}{\approx_{\ep'}}
\mu_1\big((\delta\otimes\delta)_g(y_k(d\otimes 1)y_k^*)\big)_x\\
&\stackrel{(\ref{eqn:u1flip})}{\approx_{\frac{4\ep}{9}}}
\mu_1\big(y_k(\delta_g(d)\otimes 1)y_k^*\big)_x
\stackrel{(\ref{eqn:CaseII})}{\approx_{\frac{2\ep}{9}}}
\Psi(\delta_g(d))_x.
\end{align*}
It follows that 
$\|\alpha_g^{(x)}(\Psi(d)_x)-\Psi(\delta_g(d))_x\|<\ep$,
thus establishing (E)$_x$, as desired. 

\underline{Case III:} $x \in [s-l, s+l]$.
Note that $f_1(x)=f_2(x)=0$ 
and $v_1(x)=v_2(x)=u_{1}$.
Using this at the first step and 
third steps, respectively, 
for $d\in \widetilde{F_\D}$, it follows that
\begin{align*}\label{eqn:CaseIII}\tag{4.10}
\Psi(d)_x&=
\sum_{j=1}^2 h_j(x)\sigma_j(v_j(1\otimes d\otimes 1)v_j^*)_x\\
&\stackrel{(\ref{eqn:sigmaCXmap})}{=}
\sum_{j=1}^2 h_j(x)\mu_j(v_j(x)(d\otimes 1)v_j(x)^*)_x\\
&=
\sum_{j=1}^2 h_j(x)\mu_j(u_1(d\otimes 1)u_1^*)_x\\
&\stackrel{(\ref{eqn:ufliposta})}{\approx_{\frac{\ep}{9}}}
\sum_{j=1}^2 h_j(x)\mu_j(1\otimes d)_x
\end{align*}
For $g\in K_G$ and $d\in F_\D$, and using 
at the last step that
$\delta_g(d)$ is in $\widetilde{F_\D}$, we get
\begin{align*}
\alpha_g^{(x)}(\Psi(d)_x)&\stackrel{(\ref{eqn:CaseIII})}{\approx_{\frac{\ep}{9}}}
\sum_{j=1}^2 h_j(x)\alpha_g^{(x)}(\mu_j(1\otimes d)_x)\\
&\stackrel{(\ref{eqn:mujEquiv})}{\approx_{\ep'}}
\sum_{j=1}^2 h_j(x)\mu_j(1\otimes \delta_g(d))_x
\stackrel{(\ref{eqn:CaseIII})}{\approx_{\frac{\ep}{9}}} \Psi(\delta_g(d))_x.
\end{align*}
This proves (E)$_x$,
as desired.

\underline{Case IV:} $x \in [s+l, s+2l]$. This is proved identically to Case II, by exchanging the subscripts $1$ and $2$ everywhere.

\underline{Case V:} $x \in [s+2l, s+3l]$. This is proved identically to Case I, since $\Psi(d)_x=\varphi_2(d)_x$ for all $d\in\D$. 
\end{proof}

We are now ready to prove the main result of this section, 
which is an equivariant version of Theorem~4.6 in~\cite{HirRorWin_algebras_2007}. As mentioned in the introduction, 
particular cases of it were proved in Theorem~4.28 
of~\cite{GarLup_applications_2019} and Corollary~3.7 
of~\cite{Sza_rokhlin_2019}.

\begin{thm}\label{thm:FiberToBundleHRW}
Let $G$ be a second countable, 
locally compact group, let $(\D,\delta)$ be a 
unitarily regular strongly self-absorbing $G$-algebra,
let $X$ be a second countable locally compact Hausdorff space, 
and let $(A, \alpha)$ be a separable, unital $G$-$C_0(X)$-algebra.
Assume that the covering dimension of $X$ is finite. Then $(A,\alpha)$ 
is $(\D, \delta)$-stable if and only if $(A_{x}, \alpha^{(x)})$ is $(\D, \delta)$-stable for all 
$x \in X$.
\end{thm}
\begin{proof}
The ``only if'' implication follows from Corollary~2.8 
in~\cite{Sza_stronglyII_2018} (regardless of the covering
dimension of $X$.)
We prove the ``if'' implication. 
Since $X$ is second countable
and locally compact, it follows that it is $\sigma$-compact.
Let $(U_n)_{n\in\N}$ be an increasing sequence of open subsets
of $X$ with compact closures $X_n=\overline{U_n}$, such that
$\bigcup_{n\in\N}U_n=X$. 
Set $A_n=C_0(U_n)A$, which is a $G$-invariant ideal in $A$. 
Denote by $\alpha^{(n)}$ the induced action on $A_n$.
Then $\varinjlim (A_n,\alpha^{(n)})\cong (A,\alpha)$, 
and $\pi_{X_n}$ restricts to an equivariant
isomorphism between $A_n$ and a 
$G$-invariant ideal in the $C(X_n)$-algebra $A_{X_n}$.
Since $(A_{X_n},\alpha^{(X_n)})$ is $(\D,\delta)$-absorbing 
by Corollary~2.8 in~\cite{Sza_stronglyII_2018}, and since 
direct limits of $(\D,\delta)$-absorbing $G$-algebras are
again $(\D,\delta)$-absorbing 
by Theorem~2.10 in~\cite{Sza_stronglyII_2018}, 
it suffices to prove
the statement for $X_n$ in place of $X$ and 
$(A_{X_n},\alpha^{X_n})$ in place of $(A,\alpha)$. 
Alternatively, and this is what we shall do, we may 
assume that $X$ is compact (in addition to finite-dimensional).

Using that $X$ is compact and finite-dimensional, 
together with comments preceding \autoref{key1}, we may assume that 
there is $m\in\N$ with $X=[0,1]^m$. We prove the result by
induction on $m$.

Since the case $m=0$ is trivial, let us assume that the result
holds for $m-1$. and prove it for $m$. 
Let $[0,1]^m\to [0,1]$ be the canonical projection onto the first
coordinate, and use this map to 
regard $(A, \alpha)$ as a $G$-$C([0,1])$-algebra in
such a way that the fiber over $x\in [0,1]$ is 
$A_{\{x\}\times [0,1]^{m-1}}$. By the inductive assumption, these
fibers are $(\D,\delta)$-stable. In other words, in order to 
establish the inductive step, it suffices to
prove that a $G$-$C([0,1])$-algebra is $(\D,\delta)$-stable
whenever its fibers are. Thus, we assume from now on that 
$X=[0,1]$. 

Fix finite subsets $F_\D\subseteq \D$ and $F_A\subseteq A$,
a compact subset $K_G\subseteq G$,
and tolerance $\ep>0$. 
We will produce a completely positive contractive
map $\psi\colon \D\to A$ which is
$(F_\D,F_A,K_G,\ep)$-regular for all of $[0,1]$. 
Once we do this, the result will follow
from \autoref{thm:CharactDabs} (see \autoref{rem:SufficesGoodAllX}).

Let $F_\D'\supseteq F_\D$,
and $\ep'<\ep$ be as in the 
conclusion of \autoref{key3} for $(F_\D,K_G,\ep)$.
Apply \autoref{key1} to the tuple 
$(F_\D',F_A,K_G,\ep')$
to find $n \in \mathbb{N}$, points 
$0=t_{0}<t_{1}<\dots<t_{n}=1$, and 
completely positive contractive maps $\psi_{j} \colon \D\to A$, for $j=1,\ldots,n$, such that $\psi_{j}$ is 
$(F_\D',F_A,K_G,\varepsilon')$-regular for $[t_{j-1},t_{j}]$. 

Apply \autoref{key3} to $\psi_1$ and $\psi_2$ 
to find a completely positive contractive map 
$\Psi_{2}\colon \D\to A$ which is 
$(F_\D, F_A, K_G,\varepsilon,F_\D', \varepsilon')$-regular for $[0, t_{2}]$. 
Applying \autoref{key3} to $\Psi_2$ and $\psi_3$, 
find a completely positive contractive map 
$\Psi_{3}\colon \D\to A$ which is 
$(F_\D, F_A, K_G,\varepsilon,F_\D', \varepsilon')$-regular 
for $[t_{0}, t_{3}]$.
Repeating this procedure, we arrive at a 
completely positive map $\Psi_n\colon \D\to A$
which is 
$(F_\D, F_A, K_G,\varepsilon, F_\D', \varepsilon')$-regular 
for $[0, 1]$. Thus $\Psi_n$ satisfies conditions (U), (C), (M), and (E)
from \autoref{nota:Abbreviations} with respect to $(F_\D,F_A,K_G,\ep)$, as
desired. This concludes the proof.
\end{proof}


The argument used in Proposition~4.11 of~\cite{HirRorWin_algebras_2007} 
can be adapted to our setting to show that 
for general $X$, a $G$-$C_0(X)$-algebra
$(A,\alpha)$ absorbs a unitarily regular strongly self-absorbing
$G$-algebra $(\D,\delta)$ if and only if it is 
\emph{locally} $(\D,\delta)$-absorbing (meaning that 
for every $x\in X$ there is a compact neighborhood $Y_x$ of it such that the quotient $G$-algebra
$(A_{Y_x},\alpha^{Y_x})$ is $(\D,\delta)$-stable). The non-trivial implication
follows from the fact that, under these assumptions, 
one can show along the lines of Proposition~4.9 in
\cite{HirRorWin_algebras_2007} that $(A,\alpha)$ is an iterated 
equivariant pullback of $(\D,\delta)$-stable 
$G$-algebras, and said pullbacks are $(\D,\delta)$-stable by 
Corollary~2.8 and Theorem~5.9 of~\cite{Sza_stronglyII_2018} (for which 
unitary regularity is necessary). 
This applies, in particular,
whenever $(A,\alpha)$ is equivariantly locally trivial, 
in the natural sense.

\end{document}